\documentclass[a4paper,11pt,reqno]{amsart}
\usepackage[T1]{fontenc}
\usepackage[utf8]{inputenc}
\usepackage{lmodern}
\usepackage{amsmath,amssymb,amsthm,mathtools,mathrsfs,enumerate}
\usepackage{microtype}
\microtypesetup{expansion=false}
\usepackage[hidelinks]{hyperref}
\hypersetup{pdftitle={Automorphic Form Twisted Shintani Zeta Functions over Number Fields},pdfauthor={Eun Hye Lee and Ramin Takloo-Bighash}}
\newtheorem{theorem}{Theorem}[section]
\newtheorem{proposition}[theorem]{Proposition}
\newtheorem{lemma}[theorem]{Lemma}

\theoremstyle{definition}
\newtheorem{definition}[theorem]{Definition}
\theoremstyle{remark}
\newtheorem{remark}[theorem]{Remark}
\newcommand{\A}{\mathbb A}
\newcommand{\C}{\mathbb C}
\newcommand{\R}{\mathbb R}

\newcommand{\OO}{\mathcal O}
\newcommand{\cS}{\mathscr S}
\newcommand{\cZ}{\mathcal Z}
\newcommand{\cH}{\mathcal H}
\newcommand{\one}{\mathbf 1}
\newcommand{\diag}{\operatorname{diag}}
\newcommand{\vol}{\operatorname{vol}}

\newcommand{\GL}{\mathrm{GL}}

\newcommand{\tht}{\frac7{64}}
\allowdisplaybreaks[2]
\title[Shintani Zeta Function]{Automorphic Form Twisted Shintani Zeta Function over number fields}
\author{Eun Hye Lee}
\address{Department of Mathematics, Texas Christian University, TCU Box 298900, Fort Worth, TX 76129}
\email{eun.hye.lee@tcu.edu}

\author{Ramin Takloo-Bighash}
\address{Dept. of Math, Stat, and Comp. Sci, University of Illinois at Chicago, 851 S. Morgan St, Chicago, IL 60607}
\email{rtakloo@uic.edu}

\date{}
\begin{document}
\begin{abstract}
Automorphic twists of Shintani zeta functions encode spectral information
about binary cubic forms and the shapes of cubic rings. We study their
adelic integrals over an arbitrary number field, allowing ramified data,
nonspherical vectors, and arbitrary unitary central character. For
cuspidal twists the zeta function is meromorphic on the whole plane and
holomorphic for $\Re s>1/4$.
\end{abstract}
\maketitle
\enlargethispage{8pt}

\section{Conventions and the main theorem}\label{sec:statement}

\subsection{Binary cubic forms}
Let $k$ be a number field, let $n=[k:\mathbb Q]$, and write $\A=\A_k$. Let $M_k=M_0\sqcup M_\infty$ be its places, partitioned into finite and archimedean places, and put $r=|M_\infty|=r_1+r_2$; thus $n=r_1+2r_2$. For an algebraic group or vector space $X$ and a $k$-algebra $R$, write $X_R=X(R)$ the set of $R$-points. Set $G=\GL_2$ and identify $V=k^4$ with binary cubic forms by
\[
 F_x(u,v)=x_1u^3+x_2u^2v+x_3uv^2+x_4v^3.
\]
We use the determinant twist in Wright's adelic framework \cite{Wright}. Our left action $\rho$ is defined by
\begin{equation}\label{eq:action}
 F_{g\cdot x}(u,v)=\det(g)^{-1}F_x\bigl((u,v)g\bigr). 
\end{equation}
The set up is such that a scalar matrix $aI_2$ acts on $V$ by multiplication by $a$. The discriminant is the degree-four polynomial
\[
 \operatorname{Disc}(x)=x_2^2x_3^2+18x_1x_2x_3x_4
 -4x_2^3x_4-4x_1x_3^3-27x_1^2x_4^2.
\]
It satisfies
\begin{equation}\label{eq:discriminant-character}
 \operatorname{Disc}(gx)=(\det g)^2\operatorname{Disc}(x).
\end{equation}

Let $S$ be the discriminant-zero hypersurface and put $V'=V\setminus S$.

\ 

Write $U_\A=\prod_v U_v$ for the product of the standard maximal compact subgroups of $G(k_v)$.  Let $\varphi$ be a smooth right-$U_\A$-finite automorphic cusp form on
$G(k)\backslash G(\A)$, with unitary central character $\omega$, belonging to the space of
an automorphic representation $\pi$. Thus
$\varphi(zh)=\omega(z)\varphi(h)$ for $z\in\A^\times$, viewed
as a scalar matrix. Let $\Phi\in\cS(V_\A)$ be arbitrary. 
Initially in $\Re s$ large, define
\begin{equation}
 Z(s,\Phi,\varphi)
 =\int_{G_\A/G_k}|\det g|_\A^s\varphi(g^{-1})
           \sum_{x\in V'_k}\Phi(gx)\,dg. \label{eq:Z}
           \end{equation}
This definition should be compared with the zeta functions in \cite[\S 4]{Wright}, as well as 
\cite[Eq. (72)]{HoughShape} and \cite[Eq. (3.1)]{Hough}.

In this paper we study the analytic continuation of this zeta function. We note that if $\varphi$ is a cusp form, then it is bounded, and in this case, Wright \cite[11Theorem 4.1]{Wright} implies that 
$Z(s, \Phi, \varphi)$ is absolutely
convergent for $\Re s > 2$ and it defines a holomorphic function for such $s$.  In order to outline what we prove we need some notation. For $\varphi$ and $\Phi$ as above, we set

\begin{align}
Z_+(s,\Phi,\varphi)
 &=\int_{|\det g|_\A \geq 1}|\det g|_\A^s\varphi(g^{-1})
           \sum_{x\in V'_k}\Phi(gx)\,dg, \\
 J(s,\Phi,\varphi)
 &=\int_{\substack{G_\A/G_k\\|\det g|_\A\le1}}
       |\det g|_\A^s\varphi(g^{-1})
           \sum_{x\in S_k}\Phi(gx)\,dg,\label{eq:J} \\ 
I(s,\Phi,\varphi)
& =\int_{\substack{G_\A/G_k\\|\det g|_\A\leq1}}
       |\det g|_\A^s\varphi(g^{-1})
         \sum_{x\in S_k}\Phi(g^\iota x)\,dg, 
\end{align}
where
\[
 g^\iota=(\det g)^{-1}g.
\]
Absolute convergence for $\Re s > 2$ is proved in \S\ref{sec:radial}. In the following straightforward lemma whose proof we skip we normalize the central character. For $z>0$, let $\underline z$ be the idele with finite components $1$ and every archimedean component $z^{1/n}$.

\begin{lemma}
\label{lem:central-norm-normalization}
There is a unique $\kappa\in\R$ with
$\omega(\underline z)=z^{i\kappa}$ for $z>0$. Put
\begin{equation}\label{eq:unitary-central-normalization}
 \omega_0=\omega|\cdot|_\A^{-i\kappa},\qquad
 \varphi_0(h)=|\det h|_\A^{-i\kappa/2}\varphi(h),\qquad
 \pi_0=\pi\otimes|\cdot|_\A^{-i\kappa/2}.
\end{equation}
Then for each $z >0$, $\omega_0(\underline{z}) = 1$, and $\varphi_0$ has central character $\omega_0$. We further have 
\begin{equation}\label{eq:unitary-central-shift}
 Z(s,\Phi,\varphi)=Z(s-i\kappa/2,\Phi,\varphi_0), 
\end{equation}
and similarly for $Z_+$, $I$, and $J$.  
\end{lemma}

Henceforth, unless otherwise specified, all central
characters satisfy $\omega(\underline z)=1$. We call such characters {\em normalized}. By Lemma 
\ref{lem:7}
For a cusp form $\varphi$ and 
$\Re s>2$, one has
\begin{equation}
 \begin{split}
 Z(s,\Phi,\varphi)
 ={}&Z_+(s,\Phi,\varphi)+Z_+(2-s,\widehat\Phi,\varphi^\iota)\\
 &+I(s-2,\widehat\Phi,\varphi)-J(s,\Phi,\varphi).
 \end{split}
\end{equation}
Also by Lemma \ref{lem:Z+I}, $Z_+$ and $I$ extend to entire functions of $s$. The analytic behavior of $Z$ is then controlled by $J$. The following theorem is the main result of this paper.

\begin{theorem}\label{thm:main}
 Let $\varphi$ be a smooth right-$U_\A$-finite automorphic cusp form on
$G(k)\backslash G(\A)$, with a normalized unitary central character $\omega$, belonging to the space of
an automorphic representation $\pi$. The functions $J$ and $Z$ have
unique meromorphic continuations to $\C$ and are holomorphic for
$\Re s>1/4$.
More precisely, there are finitely many pairs $(F_j,W_j)$, defined in \S \ref{sec:unfold}, with
$F_j\in\cS(\A^2)$ and $W_j$ a global lower Whittaker function for
$\pi$, an entire function $E(s)$, and meromorphic functions
$\cH_{S_0,j}$ such that
\begin{align}
 J(s,\Phi,\varphi)&=\tfrac12\sum_j A_j(s)-E(s),
                                      \label{eq:mainformula}\\
 A_j(w)&=
 \frac{L^{S_0}(3w-\tfrac12,\pi\otimes\omega^{-2})
       L^{S_0}(w+\tfrac12,\pi\otimes\omega^{-1})}
      {L^{S_0}(4w,\omega^{-2})}\,\cH_{S_0,j}(w).
                                      \label{eq:factor-main}
\end{align}
The function $A_j(w)$ is given as $\mathcal Z(F_j, W_j; 3w-1, w)$ with $\mathcal Z$ as in \S \ref{sec:local}. 
The finite set $S_0$ contains all archimedean places and all places
where the data or the additive or multiplicative characters are
ramified. Each $\cH_{S_0,j}$ is holomorphic for $\Re w>13/64$.
For an unnormalized unitary central character, the same formulas use
$(\pi_0,\omega_0,w-i\kappa/2)$ from
Lemma~\ref{lem:central-norm-normalization}.
\end{theorem}

The notation $L^{S_0}$ denotes the product of the finite Euler factors
outside $S_0$. The poles of $J$ and $Z$ are contained in the zeros of
$L^{S_0}(4s,\omega^{-2})$ and the poles of the finitely many
$\cH_{S_0,j}$. In particular, zeros $\rho$ of the Hecke denominator
give candidate poles at $s=\rho/4$. Cancellation with the numerator
or with the sum of the local factors may occur. The incomplete denominator also includes the zeros introduced by removing finitely many Euler factors, which for these unitary characters lie on $\Re(4s)=0$.

\begin{remark}
    Note that the theorem is valid for all number fields, $U_\A$-types, unitary central characters, and Schwartz functions. 
\end{remark}

\begin{remark}
If $L(w,\omega^{-2})$ has no zeros for $\Re w>\sigma_0$, then
$J$ and $Z$ are holomorphic for
\[
 \Re s>\max\{13/64,\sigma_0/4\}.
\]
For $\omega^2=1$ this is the condition on $\zeta_k$.
For a general $\omega$, the assumption concerns the indicated Hecke
$L$-function, rather than $\zeta_k$ alone.
If Ramanujan holds for $\pi$ and the generalized Riemann hypothesis
holds for $L(w,\omega^{-2})$, the local factors are holomorphic for
$\Re s>1/6$, which is therefore, conjecturally, a holomorphy region for $J$ and $Z$. We do not have any conjectures for the location of poles of $Z$ and $J$. 
\end{remark}

\subsection{Overview}
Shintani introduced zeta functions for binary cubic forms \cite{Shintani}, and Wright developed their adelic global theory \cite{Wright}. The cuspidal twist over $\mathbb Q$ was studied by Hough \cite{Hough}. His Theorem~1.1 gives holomorphy for $\Re\sigma>1/8$ in the Dirichlet-series parameter $\sigma$. By \eqref{eq:discriminant-character}, that parameter corresponds to $s=2\sigma$ here; hence the numerical boundary is $\Re s>1/4$ in our normalization. Hough's later work gives entire orbital integrals for the test functions he constructs with support in the nonsingular locus \cite[Sections~4--5]{HoughShape}. These support conditions remove the original singular contribution and do not cover arbitrary Schwartz functions. They correspond to the holomorphy of the function $I$ in our notation. 

As mentioned above, our theorem has the same holomorphy boundary as
Hough's theorem over $\mathbb Q$. Its scope includes arbitrary number
fields, Schwartz--Bruhat functions, unitary central characters, and
ramified or nonspherical vectors, together with whole-plane meromorphy.

The broader framework of prehomogeneous zeta functions is due to
Sato--Shintani \cite{SS}; automorphic-period coefficients were
developed by F.~Sato \cite{Sato94,Sato06}. The local arithmetic of
binary cubic forms is treated by Datskovsky--Wright \cite{DW2,DW3},
and the adelic method, including smoothed Eisenstein series, is
developed systematically by Yukie \cite{Yukie}. The connection with
the distribution of shapes is reflected in Bhargava--Harron
\cite{BH}, Hough \cite{HoughShape}, and the spectral estimates of
Hough--Lee \cite{HLSubconvex}. We establish the analytic continuation
needed for such twists, without deriving a counting asymptotic here.

The argument presented in this paper is indebted to several established methods. The adelic binary-cubic framework and its singular strata are developed by Wright \cite{Wright}. Hough's cuspidally twisted singular calculations \cite{Hough} motivate the use of Whittaker coefficients. We use the rapid-decay theorem of Miller--Schmid \cite[Theorem~A]{MS} to justify the completion and unfolding. A version of the Hecke series behind the unramified calculation already appears in Hough \cite[Lemma~3.5]{Hough}.

This paper owes much to Wright \cite{Wright}, Hough \cite{Hough,HoughShape}, and Hough--Lee \cite{HoughLee}. The authors thank Robert Hough, Takashi Taniguchi, and Frank Thorne for extensive discussions, and an anonymous referee for comments and for finding an error in an earlier version of this paper. The first author was partially supported by a research grant from Texas Christian University, and the second by a collaboration grant from the Simons Foundation. We benefited from conversations with ChatGPT 6 Astra.  Access to ChatGPT 6 Astra was provided through OpenAI's ChatGPT for Academic Researchers program via Julius Ross, whom the authors thank.

The paper proceeds as follows. In Section~\ref{sec:radial} we establish uniform lattice and theta estimates and describe the singular orbits. In Section~\ref{sec:unfold} we present an unfolding of the integrals. Sections~\ref{sec:local} and~\ref{sec:factorization} are dedicated to proving local meromorphy and global Euler factorization. Poisson summation completes the theorem in Section~\ref{sec:poisson}. 

\section*{Notation}

We fix a number field $k$ throughout this paper. We let $M_k$ be the set of places of $k$, $M_0$ the set of finite places, and $M_\infty$ the set of archimedean places, respectively.  For a place $v$ of $k$, we denote the corresponding completion of $k$ by $k_v$.  We let $\A = \A_k$ be the ring of adeles of $k$, and $\A^\times$ the group of ideles. We define $\A^1$ to be the set of $t \in \A^\times$ such that $|t|_\A=1$.
For each $\lambda \in \R_{>0}$ we let $\underline{\lambda} \in \A^\times$ be the idele $(t_v)$ such for each $v \mid \infty$, $t_v = \lambda^\frac{1}{n}$ and for $v \nmid \infty$, $t_v=1$.

In this paper we let $G$ be the algebraic group $\GL_2$. For a ring $R$ the set of $R$-points of $G$ is denoted by $G_R$. We will use similar notations for any algebraic subgroup of $G$.  Let $B$, $A$, and $N$ denote the algebraic subgroups of lower triangular elements of $G$, of diagonal elements of $G$, and of unipotent lower triangular elements of $G$, respectively.
Let $n(x) = \begin{pmatrix}
    1 \\ x & 1 
\end{pmatrix}$, $a(t_1, t_2) = \begin{pmatrix}
    t_1 \\ & t_2 
\end{pmatrix}$, $a(t) = a(t, t^{-1})$, and $d(t) = a(t, t)$. 

For $v\in M_\infty$, let $U_v$ be the subgroup of all unitary matrices in $G_{k_v}$. Explicitly, $U_v$ is a copy of $O(2)$, the group of $2\times 2$ orthogonal real matrices, if $k_v=\R$, or a copy of $U(2)$, the group of $2\times 2$ unitary complex matrices, if $k_v=\C$. For $v\in M_0$, let $U_v$ be $G_{\mathcal O_v}$. Then for all places $v$, $U_v$ is a maximal compact subgroup of $G_{k_v}$. Moreover, $U_\A=\prod_v U_v$ is a maximal compact subgroup of $G_\A$.

By the Iwasawa decomposition, the map $(k,b)\mapsto kb$ is a surjection of $U_\A\times B_\A$ onto $G_\A$. We may express any $g\in G_\A$, therefore, in the form $g=ka(t_1,t_2)n(c)$ with $k\in U_\A$, $t_1,t_2\in \A^\times$, and $c\in \A$. The function $t(g)=\left|\frac{t_1}{t_2}\right|_\A^{1/2}$ is independent of the choices made. Let $G_\A^1$ be the group of $g\in G_\A$ with $|\det(g)|_\A=1$. Let $B_\A^1=B_\A\cap G_\A^1$ and $B_\A^0$ be the subgroup of $b\in B_\A^1$ with $t(b)=1$. The group $B_\A^0$ is unimodular, i.e., left-invariant measures are right-invariant as well. 

For any locally compact Abelian group $X$, we define $\mathscr{S}(X)$ to be the space of Schwartz-Bruhat functions on $X$. We have three special cases that we now explain. In the first case, $X$ is isomorphic to $\R^l$ for some $l$. Pick an $\R$-basis $\{e_1, \dots, e_l\}$ for $X$. Also choose a Banach norm $\| \cdot \|$ on $X$. Given $\alpha = (\alpha_1, \dots, \alpha_l)$, define 
$$
\partial_x^\alpha = \left( \frac{\partial}{\partial x_1} \right)^{\alpha_1} \cdots \left( \frac{\partial}{\partial x_l} \right)^{\alpha_l}. 
$$
Then $\mathscr{S}(X)$ is the space of smooth functions $\Phi$ such that for all $\nu \geq 0$, 
$$
\sup_{x \in X} (1 + \|x\|^v) \sum_{|\alpha| \leq \nu} |\partial_x^\alpha \Phi(x)| < \infty. 
$$

If $X= F^l$ for some non-archimedean field $F$, then $\mathscr{S}(X)$ is defined to be the space of local constant compactly supported functions on $X$. 

If $X = \A^l$, then we define $\mathscr{S}(X)$ to be the space of finite linear combination of functions of the form 
$$
f_\infty \times \prod_{v \in M_0} f_v
$$
such that $f_\infty \in \mathscr{S}(k^l \otimes \R)$ $v \in M_0$, $f_v \in \mathscr{S}(k_v^l)$, and for almost all $v$, $f_v = \one_{\mathcal O^l}$.

\section{Radial estimates and the singular strata}\label{sec:radial}

\subsection{Measures}\label{subsec:measures}
We use normalized local absolute values, so that $|z|_{\C}=z\bar z$. Fix a nontrivial continuous additive character $\psi_k$ of $\A/k$. Additive Haar measures on $k_v$ are self-dual for $\psi_{k,v}$, and their product is denoted $du$; in particular $\vol(\A/k)=1$. At a finite place put $\vol(\OO_v^\times,d^\times t)=1$. At an archimedean place $v$, $v \in \{\R, \C\}$, write $t=r\omega$, with $r>0$ and $\omega\in\Omega_v$, where $\Omega_{\R}=\{\pm1\}$ and $\Omega_{\C}=S^1$. We take
\[
 d^\times t=e_v\frac{dr}{r}\,d\omega,
 \qquad e_{\R}=1,\quad e_{\C}=2,
\]
where $d\omega$ has total mass one. 

We use Tate's decomposition of the idele norm and Haar measure \cite[Section~4]{Tate}, with the following choice of section. Write $\A^1=\ker(|\cdot|_{\A}:\A^\times\to\R_{>0})$. For $z>0$, let $\underline z$ be the idele with finite components $1$ and every archimedean component $z^{1/n}$. Then
\begin{equation}\label{eq:section}
 |\underline z|_{\A}=z,\qquad
 \underline{z_1z_2}=\underline z_1\underline z_2.
\end{equation}
We define a multiplicative Haar measure by
\[
 d^\times t=\frac{dz}{z}\,d^\times t^1,
 \qquad t=\underline z\,t^1,\quad t^1\in\A^1, \quad z > 0.
\]
The quotient measures by $k$, $k^\times$, and $G_k$ are always those obtained from counting measure on these discrete groups.
Put
\[
 n(c)=\begin{pmatrix}1&0\\c&1\end{pmatrix},\quad
 a(t_1,t_2)=\diag(t_1,t_2),\quad a(t)=a(t,t^{-1}),\quad
 U_\A=\prod_v U_v,
\]
where $U_v$ is a standard maximal compact subgroup. Equip $U_\A$ with its probability Haar measure $du$.  We specify measure normalizations explicitly as in Wright \cite{Wright}. We fix the Haar measure on $G_\A$ by the right-Iwasawa formula
\begin{equation}\label{eq:haar}
 \int_{G_\A} f(g)\,dg
 =\frac12\int_{U_\A}\int_{(\A^\times)^2}\int_{\A}
 f\bigl(u n(c)a(a,d)\bigr)\,dc\,d^\times a\,d^\times d\,du.
\end{equation}

Let $A$ be the diagonal group, $N$ the lower unipotent group, and $B=AN$. Put $G_\A^1=\{g:|\det g|_{\A}=1\}$, $B_\A^1=B_\A\cap G_\A^1$, and $B_\A^0=\{a(t_1,t_2)n(c):t_1,t_2\in\A^1\}$. Write $d(t)=a(t,t)$. In right Iwasawa coordinates $g=u a(t_1,t_2)n(c)$, set $t(g)=|t_1/t_2|_\A^{1/2}$ and $\chi_P(g)=t(g)^2$. The factor $1/2$ in \eqref{eq:haar} gives
\begin{equation}\label{eq:radial-measure}
 dg=\frac{dz}{z}\,dg^1,\qquad g=d(\underline z)g^1.
\end{equation}

\subsection{Integrated theta bounds} 

Here we record the lattice
estimates used later in the paper.  All norms in the next lemma
are Euclidean norms.

\begin{lemma}\label{lem:uniform-lattice}
Let $L\subset\R^d$ be a full lattice.  If $Q>d$, then
\begin{equation}\label{eq:uniform-translated-lattice}
 \sup_{b\in\R^d}\sum_{\ell\in L}
       (1+R\|\ell+b\|)^{-Q}\ll_{L,Q}1+R^{-d}
       \qquad(R>0).
\end{equation}
If $A\ge0$ and $Q>d+A$, then
\begin{equation}\label{eq:uniform-nonzero-lattice}
 \sum_{0\ne\ell\in L}(1+R\|\ell\|)^{-Q}
       \ll_{L,Q,A} R^{-d}(1+R)^{-A}.
\end{equation}
Consequently, for every $f\in\mathscr S(\R^d)$ and $R>0$, if $Q > d$ then 
\[
\sup_{b \in \R^d}\sum_{\ell\in L}|f(R(\ell+b))| \ll_{L, Q}( 1 + R^{-d}) p_Q(f). 
\]
Also if $A \geq 0, Q > d +A$, 
\[
\sum_{0\ne\ell\in L}|f(R\ell)| \ll_{L, Q, A} R^{-d}(1+ R)^{-A} p_Q(f),
\]
with the Schwartz seminorm given by 
\begin{equation}\label{eq:weighted-Schwartz-seminorm}
 p_Q(f)=\sup_y(1+\|y\|)^Q|f(y)|.
\end{equation}
The latter inequalities hold uniformly for bounded subsets of Schwartz space and after
replacing $f$ by any fixed continuous differential operator with
polynomial coefficients applied to $f$.
\end{lemma}

\begin{proof} The proof has several steps. 

\ 

Step 1.  Since $L$
is a full lattice, there is a $\delta>0$ such that
$\|\ell\|\ge\delta$ for every $0\ne\ell\in L$.  Distinct points of
any translate $L+b$ are therefore at distance at least $\delta$.
Consequently, the open balls of radius $\delta/3$ about the points of
$(L+b)\cap\{\|y\|\le T\}$ are disjoint and contained in the ball
of radius $T+\delta/3$ about the origin.  Comparing Euclidean volumes
gives
\begin{equation}\label{eq:uniform-lattice-count}
 \#\{\ell\in L:\|\ell+b\|\le T\}
 \le \left(1+\frac{3T}{\delta}\right)^d
 \ll_L (1+T)^d
 \qquad(T\ge0).
\end{equation}
The implied constant is independent of $b$.

\ 

Step 2. Fix $R>0$.  The terms with $\|\ell+b\|\le R^{-1}$ each contribute at most $1$ to the sum in
\eqref{eq:uniform-translated-lattice}.  The total contribution of such terms to the sum is therefore at most the number of such terms, which by Step 1 is majorized by $1 + R^{-d}$. 

\ 

Step 3. For each integer $j\ge0$, the
terms with
\[
 2^jR^{-1}<\|\ell+b\|\le 2^{j+1}R^{-1}
\]
satisfy 
\[
(1+R\|\ell+b\|)^{-Q}
\le (1+2^j)^{-Q}
\le (2^j)^{-Q}
=2^{-jQ}.
\]
By
\eqref{eq:uniform-lattice-count}, their total contribution is majorized by 
\[
\bigl(1+2^{(j+1)d}R^{-d}\bigr)2^{-jQ}, 
\]
with the implied constant only depending on $L$. Both series
$\sum_{j\ge0}2^{-jQ}$ and
$\sum_{j\ge0}2^{-j(Q-d)}$ converge because $Q>d$.  

\ 

Steps 2 and 3 prove 
\eqref{eq:uniform-translated-lattice}, uniformly in $R$ and $b$.
We next prove \eqref{eq:uniform-nonzero-lattice}.

\ 

Step 4.  Suppose that
$A\ge0$ and $Q>d+A$.  If $0<R\le1$, the first estimate, with $b=0$,
gives
\[
 \sum_{0\ne\ell\in L}(1+R\|\ell\|)^{-Q}
 \ll_{L,Q}R^{-d}
 \le 2^A R^{-d}(1+R)^{-A}.
\]

\ 

Step 5. 
Next, we let $R\ge1$.  First,
\[
 \sum_{0\ne\ell\in L}\|\ell\|^{-Q}<\infty.
\]
  It follows that
\[
 \begin{split}
 \sum_{0\ne\ell\in L}(1+R\|\ell\|)^{-Q}
 &\le R^{-Q}\sum_{0\ne\ell\in L}\|\ell\|^{-Q}\\
 &\ll_{L,Q}R^{-d-A}
 \le 2^A R^{-d}(1+R)^{-A}.
 \end{split}
\]
Together, the two ranges of $R$ prove
\eqref{eq:uniform-nonzero-lattice}.

\ 

Step 6. For $f\in\mathscr S(\R^d)$, the definition
\eqref{eq:weighted-Schwartz-seminorm} gives
\[
 |f(R(\ell+b))|
 \le p_Q(f)(1+R\|\ell+b\|)^{-Q}.
\]
The two assertions for Schwartz functions follow by summation.
The function $p_Q$ is a continuous Schwartz seminorm: for example,
it is bounded by the usual weighted supremum seminorm with any
integer exponent at least $Q$.  Hence $p_Q$ is bounded on every
bounded subset of $\mathscr S(\R^d)$, proving the asserted uniformity.

\ 

Step 7.  Finally, let
$D=\sum_\alpha P_\alpha(y)\partial^\alpha$ be a fixed differential
operator with polynomial coefficients, with only finitely many
nonzero $P_\alpha$.  If $d_\alpha=\deg P_\alpha$, then
\[
 p_Q(Df)
 \le \sum_\alpha C_{P_\alpha}
       p_{Q+d_\alpha}(\partial^\alpha f).
\]
Thus $D$ is continuous on Schwartz space.  Applying the estimates
already proved to $Df$ establishes the last assertion, with constants
depending on $D$ but still independent of the dilation and the
translation.
\end{proof}

\begin{remark}\label{rem:adelic-lattice-uniformity}
We spell out the finite-place uniformity needed when applying the
preceding lemmas to adelic sums.  Work first on $\A^m$, with
$k_\infty^m\simeq\R^{nm}$; the application to $V_\A$ has $m=4$.
Fix a compact subset $K_f\subset\A_f^m$ and a compact open
\emph{additive} subgroup $L_f\subset\A_f^m$.  The assertion that
$\Phi$ has this finite support and translation level means
\[
 \Phi(x_f,x_\infty)=0\quad(x_f\notin K_f),\qquad
 \Phi(x_f+\ell,x_\infty)=\Phi(x_f,x_\infty)
       \quad(\ell\in L_f).
\]

Replacing $K_f$ by the compact set $K_f+L_f$ if necessary, we may
suppose that $K_f$ is a union of $L_f$-cosets.  Since $L_f$ is open,
$\A_f^m/L_f$ is discrete.  The image of $K_f$ in this discrete
quotient is compact and hence finite.  Thus, for suitable $b_\nu$,
\[
 K_f=\coprod_{\nu=1}^J(b_\nu+L_f).
\]
The corresponding finite-adelic test-function space has basis
$\one_{b_\nu+L_f}$, $1\le\nu\le J$.  Every
$\Phi\in\mathscr S(\A^m)$ with the specified support and invariance
therefore has the unique expansion
\[
 \Phi(x_f,x_\infty)
   =\sum_{\nu=1}^J
       \one_{b_\nu+L_f}(x_f)f_\nu(x_\infty),
 \qquad f_\nu\in\mathscr S(k_\infty^m).
\]

It is enough to treat  a function of the form $\one_{b_\nu+L_f}(x_f)f_\nu(x_\infty)$.  If the finite group factors range over a compact set,
their inverse images of $K_f$ lie in a single compact subset of
$\A_f^m$, hence in $\mathfrak a^{-1}\widehat{\OO}_k^m$ for some
nonzero integral ideal $\mathfrak a$.  Every contributing rational
point therefore lies in the full lattice
$(\mathfrak a^{-1})^m\subset k_\infty^m$.
The constants are controlled by finitely many seminorms of those
archimedean coefficients.  
\end{remark}

\

For $\Phi\in\mathscr S(V_\A)$ and $z>0$ put
\begin{equation}\label{eq:scalar-Schwartz-dilation}
 \Phi^z(x)=\Phi(d(\underline z)x).
\end{equation}

\begin{lemma}\label{lem:theta}
Let $\Phi\in\mathscr S(V_\A)$, and let $\varphi\in\pi$ be a smooth
$U_\A$-finite cusp form with unitary central character $\omega$, normalized by $\omega(\underline z)=1$ for $z >0$.  For every $N>0$,
\begin{equation}\label{eq:large-theta}
 \int_{G_\A^1/G_k}|\varphi(g^{-1})|
       \sum_{0\ne x\in V_k}|\Phi^z(gx)|\,dg^1
       \ll_{N,\Phi,\varphi}z^{-N}\qquad(z\ge1).
\end{equation}
For $0<z\le1$ the same integral is $O_{\Phi,\varphi}(z^{-4})$.
Both assertions hold after a fixed archimedean right differential
operator is applied to $\varphi$, or a fixed power of $z\partial_z$ is
applied to $\Phi^z$.  If we fix $K_f$ and $L_f$ as in Remark \ref{rem:adelic-lattice-uniformity} the implied constants involve only finitely many archimedean Schwartz seminorms of $\Phi$.
\end{lemma}

\begin{proof}
Reduction theory for $\GL_2$ \cite[Section~10]{JL}, together with the
compactness of $\A^1/k^\times$ and $\A/k$
\cite[Section~4]{Tate}, gives relatively compact representative sets
$C_1\subset\A^1$, $C_N\subset\A$, and $\lambda_0>0$ such that
\begin{equation}\label{eq:explicit-Siegel-set}
 \mathfrak S=\{u a(\underline\lambda)a(t_1,t_2)n(c):
      u\in U_\A,\ \lambda\ge\lambda_0,
      \ t_1,t_2\in C_1,\ c\in C_N\}
\end{equation}
maps onto $G_\A^1/G_k$.  For a nonnegative quotient function its integral
is bounded by a fixed multiple of the corresponding integral on this
set with measure 
\begin{equation}\label{eq:explicit-Siegel-density}
 du\,\lambda^{-2}\frac{d\lambda}{\lambda}\,
 d^\times t_1\,d^\times t_2\,dc.
\end{equation}

By Remark~\ref{rem:adelic-lattice-uniformity}, it suffices to consider
one archimedean Schwartz coefficient and one fixed lattice
$\Lambda\subset V_{k_\infty}\simeq\R^{4n}$ containing all contributing
rational vectors.  Write $H=\max(1,\lambda)$.  The diagonal action is
\[
 \rho(a(\underline\lambda))x
  =(\underline\lambda^3x_1,\underline\lambda x_2,
       \underline\lambda^{-1}x_3,\underline\lambda^{-3}x_4).
\]
All remaining factors and their inverses range over compact sets, so
\begin{equation}\label{eq:explicit-action-bounds}
 \|\rho(g_\infty)\|_{\rm op}
       +\|\rho(g_\infty)^{-1}\|_{\rm op}
       \ll H^{3/n}\qquad(g\in\mathfrak S).
\end{equation}
In particular, $\|\rho(g_\infty)x\|\ge cH^{-3/n}\|x\|$ for some
fixed $c>0$.  Put $r=z^{1/n}$ and $R=crH^{-3/n}$.
Lemma~\ref{lem:uniform-lattice}, with $d=4n$, gives
\[
 \sum_{0\ne x\in V_k}|\Phi^z(gx)|
 \ll p(\Phi_\infty)R^{-4n}(1+R)^{-A},
\]
where the order of $p$ may depend on $A$.  For $z\ge1$, choose
$A=nN$ and use $(1+R)^{-nN}\le R^{-nN}$.  Since $r^{-4n}\le1$,
this yields
\begin{equation}\label{eq:pointwise-theta}
 \sum_{0\ne x\in V_k}|\Phi^z(gx)|
       \ll_{N,\Phi}z^{-N}H^{12+3N}.
\end{equation}
For $0<z\le1$, taking $A=0$ gives
\begin{equation}\label{eq:pointwise-theta-small}
 \sum_{0\ne x\in V_k}|\Phi^z(gx)|
       \ll_\Phi z^{-4}H^{12}.
\end{equation}

Cuspidal rapid decay \cite[Theorem~A]{MS} applied to adelic automorphic forms as indicated by the authors gives, for every $M>0$,
\begin{equation}\label{eq:explicit-cusp-decay}
 |\varphi(g^{-1})|\ll_{M,\varphi}H^{-M}
                         \qquad(g\in\mathfrak S).
\end{equation}
Combining this with \eqref{eq:explicit-Siegel-density} and
\eqref{eq:pointwise-theta}, the height integral for $\lambda\ge1$ is
bounded by
\[
 C_{N,\Phi,\varphi}z^{-N}
       \int_1^\infty\lambda^{12+3N-M-2}\frac{d\lambda}{\lambda}
 =\frac{C_{N,\Phi,\varphi}z^{-N}}{M-10-3N}
\]
when $M>10+3N$.  The small-radius bound follows similarly with $M>10$.
The bounded interval $\lambda_0\le\lambda\le1$, when nonempty,
contributes the same respective bounds.

A fixed archimedean right derivative of a smooth $U_\A$-finite
cuspidal vector remains a smooth cuspidal vector to which the same
rapid-decay estimate applies.  Finally, $z\partial_z$ acts on
$\Phi^z$ by $n^{-1}$ times the Euclidean Euler operator on
$V_{k_\infty}$.  Its fixed powers are continuous polynomial
differential operators on Schwartz space.  The seminorm assertions
therefore follow from Lemma~\ref{lem:uniform-lattice} and
Remark~\ref{rem:adelic-lattice-uniformity}.
\end{proof}

\subsection{The two nonzero singular orbits}
Wright \cite{Wright} describes the singular strata.   A nonzero singular binary cubic
over a characteristic-zero field has either a triple root or a double
root and a distinct simple root.  In either case each root is rational.  Thus $S_k$ is the disjoint union of
$\{0\}$ and two $G_k$-orbits, represented by
\[
 x_1=(0,0,0,1),\qquad x_2=(0,0,1,0).
\]
Their stabilizers are
\begin{align}
 H_1(k)&=\left\{\begin{pmatrix}d^2&0\\c&d\end{pmatrix}:
                    d\in k^\times,\ c\in k\right\},\label{eq:H1}\\
 H_2(k)&=\{\diag(a,1):a\in k^\times\}.\label{eq:H2}
\end{align}
Indeed, preservation of the triple root forces a lower triangular
matrix, which sends $v^3$ to $(d^2/a)v^3$.  Preservation of the two roots
with distinct multiplicities forces a diagonal matrix, which sends
$uv^2$ to $d\,uv^2$.  We have 
\begin{equation}\label{eq:orbits}
 S_{i,k}=G_k\cdot x_i\simeq G_k/H_i(k)\qquad(i=1,2).
\end{equation}
This parametrization is equivalent to Wright's.

For $\varphi$ and $\Phi$ as in the introduction, define
\begin{equation}\label{eq:singular-fixed-radius}
 J^1(\Phi,\varphi)=
 \int_{G_\A^1/G_k}\varphi(g^{-1})
                   \sum_{x\in S_k}\Phi(gx)\,dg^1.
\end{equation}
Also let 
\[ 
R_{\Phi,\varphi}(z)=J^1(\Phi^z,\varphi).
\]
The radial measure formula \eqref{eq:radial-measure} then gives, at first formally, 
\begin{equation}\label{eq:J-radial}
 J(s,\Phi,\varphi)=\int_0^1z^{2s}R_{\Phi,\varphi}(z)\,\frac{dz}{z}.
\end{equation}

\begin{lemma}\label{lem:abscon-J}
    The integral $J(s, \Phi, \varphi)$ converges absolutely for $\Re s> 2$. 
\end{lemma}

\begin{proof}
Write \(X=\) and $C_\varphi=\int_{G_\A^1/G_k}|\varphi(g^{-1})|\,dg^1<\infty$. As \(\Phi^z(0)=\Phi(0)\),
\[
\begin{aligned}
A(z)
&:=\int_{G_\A^1/G_k}|\varphi(g^{-1})|
           \sum_{x\in S_k}|\Phi^z(gx)|\,dg^1\\
&=|\Phi(0)|\int_{G_\A^1/G_k}|\varphi(g^{-1})|\,dg^1
 +\int_{G_\A^1/G_k}|\varphi(g^{-1})|
           \sum_{x\in S_k\setminus\{0\}}|\Phi^z(gx)|\,dg^1.
\end{aligned}
\]
The second integral is bounded by
\[
\int_{G_\A^1/G_k}|\varphi(g^{-1})|
       \sum_{0\ne x\in V_k}|\Phi^z(gx)|\,dg^1
\le C_{\Phi,\varphi}z^{-4},
\]
by Lemma \ref{lem:theta}, since \(S_k\setminus\{0\}\subset V_k\setminus\{0\}\). Hence
\[
A(z)\le C_\varphi|\Phi(0)|+C_{\Phi,\varphi}z^{-4}.
\]

Finally, with \(\sigma=\Re s>2\),
\[
\int_0^1z^{2\sigma}A(z)\,\frac{dz}{z}
\le
\frac{C_\varphi|\Phi(0)|}{2\sigma}
+\frac{C_{\Phi,\varphi}}{2\sigma-4}<\infty.
\]
\end{proof}
Now we use the structure of the singular set to simplify $J$. 
As cusp forms are orthogonal to the trivial representation, we have 
\begin{equation}\label{eq:cusp-zero-average}
 \int_{G_\A^1/G_k}\varphi(g^{-1})\,dg^1=0.
\end{equation}
The $S_{1,k}$-contribution also vanishes. Indeed,
\[
 S_{1,k}=\coprod_{\gamma\in G_k/B_k}
 \gamma\{(0,0,0,\beta):\beta\in k^\times\}.
\]
Unfolding this disjoint parametrization gives an integral over
$G_\A^1/B_k$. In the coordinates
$g=u a(\underline\lambda)a(t_1,t_2)n(c)$, its theta function is
\[
 \sum_{\beta\in k^\times}
 \Phi^z\bigl(u(0,0,0,
       \underline\lambda^{-3}t_2^2t_1^{-1}\beta)\bigr),
\]
which is independent of $c$. The integral of
$\varphi(n(-c)a(t_1,t_2)^{-1}a(\underline\lambda)^{-1}u^{-1})$
over $c\in\A/k$ vanishes by cuspidality. Lemma~\ref{lem:theta}
justifies unfolding and this integration by absolute convergence.

The preceding vanishing statements show that only $S_{2,k}$ contributes to $J^1$ and $R$. 
We have 
\begin{equation}\label{eq:R}
 R_{\Phi,\varphi}(z)=J^1(\Phi^z,\varphi)=
 \int_{G_\A^1/G_k}\varphi(g^{-1})
                   \sum_{x\in S_{2,k}}\Phi^z(gx)\,dg^1.
\end{equation}

\begin{proposition}\label{prop:tail}
The function
\begin{equation}\label{eq:E}
 E(s)=\int_1^\infty z^{2s}R_{\Phi,\varphi}(z)\,\frac{dz}{z}
\end{equation}
is entire.  The function
\begin{equation}\label{eq:M}
 M(s)=\int_0^\infty z^{2s}R_{\Phi,\varphi}(z)\,\frac{dz}{z}
\end{equation}
converges absolutely and is holomorphic for $\Re s>2$. We have 
\begin{equation}\label{eq:MJE}
 J(s,\Phi,\varphi)=M(s)-E(s).
\end{equation}
Replacing $S_{2,k}$ by any $G_k$-stable subset of
$V_k\setminus\{0\}$ in integral \eqref{eq:E} too gives
an entire function.
\end{proposition}

\begin{proof}
Let
\[
 T(z)=\int_{G_\A^1/G_k}|\varphi(g^{-1})|
                       \sum_{0\ne x\in V_k}|\Phi^z(gx)|\,dg^1.
\]
Clearly $|R_{\Phi,\varphi}(z)|\le T(z)$. Lemma~\ref{lem:theta} gives 
\begin{equation}\label{eq:Mellin-endpoint-bounds}
 T(z)\ll
 \begin{cases}
 z^{-4},&0<z\le1,\\
 z^{-N},&z\ge1,
 \end{cases}
\end{equation}
for every $N>0$ with the implied constant depending on $N$ for $ z \geq 1$.  This immediately gives the holomorphy of $E(s)$ on $\C$. 

The holomorphy of $M(s)$ and $J$ for $\Re s > 2$ follows from Lemma \ref{lem:abscon-J}. 

Finally, if $Y\subset V_k\setminus\{0\}$ is $G_k$-stable, its theta
sum descends to $G_\A^1/G_k$, and the absolute value of its cuspidal
pairing is bounded by the same $T(z)$.  The preceding large-radius
argument applies without change.  The substitution
$g=d(\underline z)h$ and $\omega(\underline z)=1$ identify
that Mellin integral with
\[
 \int_{\substack{G_\A/G_k\\|\det g|_\A\ge1}}
 |\det g|_\A^s\varphi(g^{-1})\sum_{x\in Y}\Phi(gx)\,dg.
\]
This proves the final assertion.
\end{proof}

\section{The completed Whittaker unfolding}\label{sec:unfold}

\subsection{Whittaker conventions and standard local facts}\label{subsec:whittaker-facts}
Throughout Sections~\ref{sec:unfold}--\ref{sec:factorization}, let
$\varphi$ be a smooth right-$U_\A$-finite cusp form in an irreducible
cuspidal representation $\pi$ with unitary central character $\omega$.
We first assume $\omega(\underline z)=1$ for $z>0$; the unitary norm
component is restored by the vertical parameter shift described in
the introduction.
All character twists of a representation below are through its
determinant. Define the lower
Whittaker function by
\begin{equation}\label{eq:Whittaker}
 W_\varphi(h)=\int_{\A/k}\varphi(n(c)h)\psi_k(c)\,dc.
\end{equation}
We have 
\[
 W_\varphi(n(b)h)=\psi_k(-b)W_\varphi(h).
\]
The standard (absolutely convergent) Fourier expansion is
\begin{equation}\label{eq:Fourier-expansion}
 \varphi(h)=\sum_{\alpha\in k^\times}
                     W_\varphi(a(\alpha^{-1},1)h); 
\end{equation}
see Jacquet--Langlands \cite[Section~11]{JL}.

The Kirillov coordinate throughout the paper is
$t\mapsto W(a(1,t))$.  To compare it with the usual upper-unipotent
convention, put $w_0=\left(\begin{smallmatrix}0&1\\1&0\end{smallmatrix}\right)$
and $\widetilde W(h)=W(w_0hw_0)$.  If
$n^+(c)=w_0n(c)w_0$, then
\[
 \widetilde W(n^+(c)h)=W(n(c)w_0hw_0)
                    =\psi_k(-c)\widetilde W(h),
\]
and
\begin{equation}\label{eq:Kirillov}
 \widetilde W(a(t,1))=W(a(1,t))
     =\omega(t)W(a(t^{-1},1)).
\end{equation}
The last equality is the central transformation law.
Conjugation by $w_0$ is an inner automorphism, so this is a Whittaker
realization of the same representation $\pi$, with additive character
$\psi_k^{-1}$.  Consequently its standard local $L$-factors are
$L(z,\pi_v)$, without a change to a contragredient representation.
The usual Kirillov estimates and scalar Mellin theorem apply to exactly
$W(a(1,t))$.  We use the local theory of Jacquet--Langlands
\cite[Sections~2, 5--6]{JL}.

We record explicitly the standard local results required below. Let $K$ be a local field. Write $\pi_K$ for an irreducible unitary generic representation of $\GL_2(K)$ with unitary central character $\omega_K$, and let $\psi_K$ denote the additive character whose inverse defines its lower Whittaker model. Put $\theta_K=0$ in the tempered case; if \(\pi_K\) is non-tempered, write its normalized principal-series realization as\[
\pi_K\simeq
\operatorname{Ind}_{B(K)}^{\mathrm{GL}_2(K)}
\left(\mu_1|\cdot|_K^{r}\otimes
      \mu_2|\cdot|_K^{-r}\right),
\qquad r>0,
\]
where \(\mu_1,\mu_2\) are unitary characters. Then set $\theta_K =r$.
For a fixed maximal-compact-finite smooth Whittaker vector $W$, for every $\epsilon>0$,
\begin{equation}\label{eq:W-small}
 |W\bigl(a(1,t)\bigr)|\ll_{W,\epsilon}|t|_K^{1/2-\theta_K-\epsilon}
       \qquad(0<|t|_K\leq1).
\end{equation}
At an archimedean place, $t\mapsto W(a(1,t))$ decreases faster than
every power at infinity.  At a nonarchimedean place it vanishes for
$|t|_K$ sufficiently large.  

For local components of automorphic cuspidal representations, Blomer--Brumley \cite[Definition~1 and Theorem~1]{BB} give, including at ramified and archimedean places,
\begin{equation}\label{eq:theta}
 \theta_{k_v}\leq\theta:=\tht.
\end{equation}
Their theorem allows arbitrary unitary central character.

Local Hecke theory also supplies the following scalar Mellin theorem, due to Jacquet and Langlands \cite{JL}.  If $\xi$ is a fixed unitary character of $K^\times$, then
\begin{equation}\label{eq:scalar-local-mellin}
 \int_{K^\times}W\bigl(a(1,t)\bigr)\xi(t)|t|_K^{z-1/2}\,d^\times t
\end{equation}
has meromorphic continuation in $z$, and its quotient by $L(z,\pi_K\otimes\xi)$ is entire; see \cite[Theorem~2.18(ii)--(iii) for nonarchimedean fields, and Sections~5--6 for archimedean fields]{JL}. 

\subsection{Compact averaging}
Suppose $\varphi$ is a right $U_\A$-finite cusp form. This means that  the space spanned by the right $U_\A$-translates of $\varphi$ is finite dimensional. Choose a basis $\varphi_1,\ldots,\varphi_m$ for the space spanned by the right $U_\A$-translates of $\varphi$ and write
\begin{equation}\label{eq:K-type}
 \varphi(hu^{-1})=\sum_{j=1}^m c_j(u)\varphi_j(h), \qquad (u \in U_\A).
\end{equation}
Define 
\begin{align}
 \Phi_{j}(x)&=\int_{U_\A} c_j(u)\Phi(ux)\,du,\label{eq:Psi}\\
 F_j(x,y)&=\widehat\Phi_{j,4}(0,0,x,y):=\int_{\A}\Phi_{j}(0,0,x,v)\psi_k(yv)\,dv.\label{eq:F}
\end{align}
Both compact averaging and partial Fourier transformation preserve the Schwartz--Bruhat space.

\begin{lemma}\label{lem:domination}
Let $P\geq0$ be measurable on $\A$. Assume that there is a compact
$\Omega_f\subset\A_f$ such that, for every $N\geq0$, there is a $C_N >0$ such that 
\begin{equation}\label{eq:major-P}
 P(x_f,x_\infty)\leq C_N\one_{\Omega_f}(x_f)
                       (1+\|x_\infty\|)^{-N}.
\end{equation}
Let $W$ be the global Whittaker function of a smooth right-$U_\A$-finite
cusp form in $\pi$. For every real $\sigma>1$,
\begin{equation}\label{eq:majorant}
 \int_{(\A^\times)^2}P(x)|W(a(1,xy))|
 |x|_\A^{3\sigma-1}|y|_\A^\sigma\,d^\times x\,d^\times y<\infty.
\end{equation}
Convergence is uniform on compact subintervals of $\sigma>1$, including
after multiplication by any fixed nonnegative powers of
$|\log|x|_\A|$ and $|\log|y|_\A|$.
\end{lemma}
\begin{proof}
Make a change of variables $t=xy$ identify the integral with
\begin{equation}\label{eq:majorant-product}
 \left(\int_{\A^\times}P(x)|x|_\A^{2\sigma-1}\,d^\times x\right)
 \left(\int_{\A^\times}|W(a(1,t))||t|_\A^\sigma\,d^\times t\right).
\end{equation}
For the first factor, enlarge $\Omega_f$ to a product of compact open
subsets, equal to $\OO_v$ outside a finite set. Its unramified factors
are $(1-q_v^{-(2\sigma-1)})^{-1}$; their product converges because
$2\sigma-1>1$.  If we apply \eqref{eq:theta} the second integral converges for $\Re s > 39/64$. This finishes the proof of the lemma. 
\end{proof}

For $F\in\cS(\A^2)$ and a global lower Whittaker function $W$, set
\begin{equation}\label{eq:global-double}
 \cZ(F,W;u_1,u_2)=
 \begin{aligned}[t]
 &\int_{(\A^\times)^2}F(x,y)W\bigl(a(1,xy)\bigr)\\
 &\quad\times\omega(x)^{-2}\omega(y)^{-1}
        |x|_{\A}^{u_1}|y|_{\A}^{u_2}\,d^\times x\,d^\times y.
 \end{aligned}
\end{equation}
The central character of $W$ is understood in the notation $\cZ$. Initially this is defined in a right chamber of absolute convergence. The function $P(x)=\sup_y|F(x,y)|$ satisfies \eqref{eq:major-P} of Lemma~\ref{lem:domination}, which therefore applies at $u_1=3s-1$, $u_2=s$ for $\Re s>2$.

The next calculation is motivated by Hough's cuspidally twisted singular integral \cite[Section~3]{Hough}. We complete the radial Mellin integral before unfolding and use the stabilizer $H_2$ of the fixed representative. 

\begin{proposition}\label{prop:unfold}
For $\Re s>2$, with notation as in \eqref{eq:K-type} and \eqref{eq:F},  one has
\begin{equation}\label{eq:uncut}
 M(s)=\frac12\sum_{j=1}^m\cZ(F_j,W_{\varphi_j};3s-1,s).
\end{equation}
\end{proposition}

\begin{proof}
By Proposition \ref{prop:tail} for  $\sigma=\Re s>2$, the integral defining $M(s)$ is absolutely convergent. 
Under $g=d(\underline z)h$ we have $dg=dz/z\,dh$,
$|\det g|=z^2$, and $\varphi(g^{-1})=\varphi(h^{-1})$ as we are assuming that $\omega(\underline{z}) =1$ for $z > 0$. The stabilizer of $x_2=(0,0,1,0)$ is
$H_2(k)=\{a(q,1):q\in k^\times\}$.
Thus $S_{2,k}=G_k/H_2(k)$, and 
unfolding gives
\begin{equation}\label{eq:unfold-orbit-integral}
 M(s)=\int_{G_\A/H_2(k)}|\det g|^s
                    \varphi(g^{-1})\Phi(gx_2)\,dg.
\end{equation}

Write $g=u n(c)a(a,d)$. Right multiplication by $H_2(k)$ changes
only $a$ to $aq$, so \eqref{eq:haar} gives the quotient measure
$\tfrac12du\,dc\,d^\times a\,d^\times d$ with
$a\in\A^\times/k^\times$ and $d\in\A^\times$. Since
$n(c)a(a,d)x_2=(0,0,d,cd)$, compact averaging and $v=cd$ give
\begin{equation}\label{eq:unfold-additive-change}
 \begin{aligned}
 M(s)=\frac12\sum_j\int_{\A^\times/k^\times}
 \int_{\A^\times}\int_\A
 &|a|^s|d|^{s-1}\Phi_j(0,0,d,v)\\
 &\times\varphi_j(a(a^{-1},d^{-1})n(-v/d))
 \,dv\,d^\times d\,d^\times a.
 \end{aligned}
\end{equation}
The additive Jacobian is $dc=|d|^{-1}dv$.
The Whittaker expansion of $\varphi_j$ has terms
$W_j(a(\alpha^{-1},1)h)$, $\alpha\in k^\times$. The identities
$W_j(n(b)h)=\psi_k(-b)W_j(h)$ and
$W_j(d(t)h)=\omega(t)W_j(h)$ give
\begin{equation}\label{eq:unfold-frequency-term}
 \begin{aligned}
 W_j(a((a\alpha)^{-1},d^{-1})n(-v/d))
 =\psi_k(a\alpha v/d^2)\omega(a\alpha)^{-1}
                   W_j(a(1,a\alpha/d)).
 \end{aligned}
\end{equation}

Put $P_j(d)=\int_\A|\Phi_j(0,0,d,v)|\,dv$; its Schwartz
bounds imply the hypothesis of Lemma~\ref{lem:domination}.
Tonelli's theorem unfolds the sum of the absolute Fourier terms
from $a\in\A^\times/k^\times$ to $a\in\A^\times$.
The change $x=d$, $y=a/d^2$ identifies their integral with
\begin{equation}\label{eq:unfold-double-majorant}
 \int_{(\A^\times)^2}P_j(x)|W_j(a(1,xy))|
       |x|^{3\sigma-1}|y|^\sigma\,d^\times x\,d^\times y,
\end{equation}
which is finite by that lemma. We may therefore insert the Fourier
expansion in \eqref{eq:unfold-additive-change} and unfold it.
The same change of variables preserves product multiplicative
Haar measure and gives
\[
 |a|^s|d|^{s-1}=|x|^{3s-1}|y|^s,\qquad
 \omega(a)^{-1}=\omega(x)^{-2}\omega(y)^{-1},\qquad
 av/d^2=yv.
\]
The additive integral is exactly $F_j(x,y)$ from \eqref{eq:F}.
This proves \eqref{eq:uncut}, with the factor $1/2$ inherited
from \eqref{eq:haar}. It also proves absolute convergence of the
final integrals, since $|F_j(x,y)|\le P_j(x)$.
\end{proof}

\section{Local meromorphic distributions}\label{sec:local}

For a local field $K$, $f\in\cS(K^2)$, and a local Whittaker function $W$, write
\[
 \begin{aligned}
 \cZ_K(f,W;u_1,u_2)={}&\int_{(K^\times)^2}
 f(x,y)W\bigl(a(1,xy)\bigr)\\
 &\quad\times\omega_K(x)^{-2}\omega_K(y)^{-1}
 |x|_K^{u_1}|y|_K^{u_2}\,d^\times x\,d^\times y
 \end{aligned}
\]
in its region of absolute convergence. For the remainder of the paper put
\begin{equation}\label{eq:localT}
 T_{W,f}(w)=\cZ_K(f,W;3w-1,w).
\end{equation}
At a place $v$ we abbreviate $\cZ_{k_v}$ to $\cZ_v$.
The local representation is irreducible, unitary, and generic, with unitary central character $\omega_v$. Its smooth Whittaker vector is fixed throughout; at any archimedean place $v$ it is $U_v$-finite.

\begin{definition}\label{def:meromorphic-distribution}
A family $T(w): f \mapsto T_f(w)$ of distributions on $\cS(K^2)$ parametrized by $w \in \C$ is called {\em meromorphic} if the following two conditions hold: 

\begin{enumerate}[(i)]
\item Every $w_0 \in \C$ has a neighborhood $\Omega$ on which there exists a nonzero holomorphic function $p(w)$, independent of $f$, such that $p(w)T_f(w)$ is holomorphic.  

\item  Let $w_0, \Omega, p(w)$ be as in (i).
If $K$ is archimedean, we require that on each compact subset $Q$ of this $\Omega$ there be a continuous Schwartz seminorm $q_Q$ and a scalar $C_Q$ such that
\begin{equation}\label{eq:seminorm-bound}
 \sup_{w\in Q}|p(w)T_f(w)|\leq C_Q q_Q(f).
\end{equation}
Next let \(K\) be nonarchimedean, with ring of integers \(\mathcal O_K\), uniformizer \(\varpi\), and residue cardinality \(q\). For integers \(C,M\geq0\), define
\[
\mathcal S_{C,M}
=
\left\{
f:K^2\to\mathbb C:
\begin{array}{l}
\operatorname{supp}f\subseteq\varpi^{-C}\mathcal O_K^2,\\
f(v+h)=f(v)\quad
\text{for all }v\in K^2,\ h\in\varpi^M\mathcal O_K^2
\end{array}
\right\}.
\]
For every fixed \(C,M\) and every compact parameter set \(Q\) in the neighborhood under consideration, there exists a constant \(A_{Q,C,M}\) such that
\[
 \sup_{w\in Q}|p(w)T_{f}(w)|
 \leq A_{Q,C,M}\,\|f\|_\infty, 
 \qquad f\in \cS_{C,M}.
 \]
 \end{enumerate}
\end{definition}

\begin{proposition}\label{prop:local}
For fixed $f$ and $W$ the two-variable local integral $\cZ_K(f,W;u_1,u_2)$ is absolutely convergent and holomorphic in
\begin{equation}\label{eq:localtube}
 \Re u_1>-\frac12+\theta_K,\qquad
 \Re u_2>-\frac12+\theta_K.
\end{equation} 
For fixed $W$, the family of distributions $f \mapsto T_{W,f}(w)$ has a meromorphic continuation to $\C$ in the sense of Definition~\ref{def:meromorphic-distribution}. If $K$ is nonarchimedean, it is rational in $q_K^{-w}$. 
\end{proposition}

The proof of this proposition is presented in the next three subsection. The first subsection establishes the absolute convergence and holormophy in the domain \eqref{eq:localtube}, the subsequent two subsections prove meromorphic continuation.

\subsection{Proof of Proposition \ref{prop:local} in the nonarchimedean case}\label{sec:nonarc_denom}
Write
\[
 \kappa(t)=W(a(1,t)),\qquad
 U=\mathcal O_K^\times,\qquad
 q=|\mathcal O_K/\varpi\mathcal O_K|,\qquad
 c_\omega=\omega_K(\varpi).
\]
Thus $|\varpi|_K=q^{-1}$, $|c_\omega|=1$, and
$\operatorname{vol}(U,d^\times t)=1$. 

To establish absolute convergence and holomorphy in the domain \eqref{eq:localtube} we proceed as follows. Because the character weights have absolute value one, it remains to prove the same absolute convergence and holomorphy asserted in \eqref{eq:localtube}. On a compact subset of \eqref{eq:localtube}, choose $\epsilon>0$ smaller than both strict margins. Where $|xy|_K\leq1$, \eqref{eq:W-small} bounds the integrand by a constant times
\[
 |f(x,y)|\,
 |x|_K^{\Re u_1+1/2-\theta_K-\epsilon}
 |y|_K^{\Re u_2+1/2-\theta_K-\epsilon}.
\]
Both exponents are positive. This is integrable near either axis and at their intersection; away from the axes compact support or Schwartz decay absorbs the positive weights. Where $|xy|_K\geq1$, a nonarchimedean test function restricts both valuations to a finite range, using also the upper support bound for $t\mapsto W(a(1,t))$. Logarithmic factors are absorbed by
shrinking the margins, so the integrals and their parameter derivatives converge normally. This proves local holomorphy with continuous test-function
bounds. 

Next we establish the remaining statements of the proposition in the nonarchimedean setting. We recall an asymptotic formula for $W$.  There exist an integer $\ell_0\geq0$, nonzero complex numbers
$\alpha_1,\ldots,\alpha_r$, nonnegative integers $d_1,\ldots,d_r$,
and locally constant functions $c_{a,j}:U\to\C$ such that
\begin{equation}\label{eq:finite-kirillov-germ}
 \kappa(\varpi^\ell u)
 =\sum_{a=1}^r\sum_{j=0}^{d_a}
       c_{a,j}(u)\alpha_a^\ell\ell^j
 \qquad(\ell\geq\ell_0,\ u\in U).
\end{equation}
This is an exact equality on a punctured neighborhood of zero,
uniformly for $u\in U$. This is well-known and already implicit in \cite{JL}. A convenient reference is \cite[Theorem~2.1]{Matringe}.

We next prove absolute convergence in \eqref{eq:localtube}.
Choose $C\geq0$ with
$\operatorname{supp}f\subseteq\varpi^{-C}\mathcal O_K^2$.
Write $x=\varpi^m\xi$ and $y=\varpi^n\eta$, with $\xi,\eta\in U$;
only $m,n\geq-C$ can contribute. The pairs with $m+n<0$ form a
finite set, and their integrals over $U^2$ are entire functions of
$(u_1,u_2)$. Let $Q$ be a compact subset of \eqref{eq:localtube}
and put
\[
 b_i=\min_{(u_1,u_2)\in Q}\Re u_i\quad(i=1,2).
\]
Choose $\epsilon>0$ so that
$\delta_i=b_i+\tfrac12-\theta_K-\epsilon>0$ for $i=1,2$.
For $m+n\geq0$, the estimate \eqref{eq:W-small} gives
\[
 |\kappa(\varpi^{m+n}\xi\eta)|
 \ll_{W,\epsilon}
 q^{-(m+n)(1/2-\theta_K-\epsilon)}.
\]
Since the character weights have absolute value one and every
multiplicative shell has measure one, the sum of the absolute
integrals over these shells is, uniformly on $Q$, at most
\[
 C_{Q,C,W}\|f\|_\infty
 \sum_{m,n\geq-C}
 q^{-\delta_1(m+C)}q^{-\delta_2(n+C)}<\infty.
\]
This proves absolute convergence and joint
holomorphy there.

We now use \eqref{eq:finite-kirillov-germ} to prove rationality
with a denominator independent of $f$. Choose $M_0\geq0$ so that
$f$ is invariant under translation by
$\varpi^{M_0}\mathcal O_K^2$, and choose an integer
\[
 M\geq\max\{M_0,C+\ell_0,0\}.
\]
Put
\[
 z=q^{-w},\qquad
 \Xi(\xi,\eta)=\omega_K(\xi)^{-2}\omega_K(\eta)^{-1}.
\]
For $\Re w$ sufficiently large, the absolute convergence just
proved gives 
\begin{equation}\label{eq:finite-valuation-sum}
 \begin{aligned}
 T_{W,f}(w)=\sum_{m,n\geq-C}
 q^m c_\omega^{-2m-n}z^{3m+n}
 \int_{U^2}&f(\varpi^m\xi,\varpi^n\eta)\\
 &\times\kappa(\varpi^{m+n}\xi\eta)\Xi(\xi,\eta)
 \,d^\times\xi\,d^\times\eta.
 \end{aligned}
\end{equation}
We partition its valuation
indices into four disjoint regions:
\[
 \begin{aligned}
 \mathrm{I}&:\quad -C\leq m<M,\quad -C\leq n<M,\\
 \mathrm{II}&:\quad m\geq M,\quad -C\leq n<M,\\
 \mathrm{III}&:\quad -C\leq m<M,\quad n\geq M,\\
 \mathrm{IV}&:\quad m\geq M,\quad n\geq M.
 \end{aligned}
\]
Let $\mathcal T_{\mathrm I}(w),\ldots,\mathcal T_{\mathrm{IV}}(w)$
denote the contributions obtained by restricting
\eqref{eq:finite-valuation-sum} to these regions. Thus, initially
for $\Re w$ sufficiently large,
\[
 T_{W,f}(w)=\mathcal T_{\mathrm I}(w)
 +\mathcal T_{\mathrm{II}}(w)
 +\mathcal T_{\mathrm{III}}(w)
 +\mathcal T_{\mathrm{IV}}(w).
\]
Recall that
$M\geq\max\{M_0,C+\ell_0,0\}$ and that $f$ is invariant under
translation by $\varpi^{M_0}\mathcal O_K^2$.

\medskip
\noindent\emph{Region I: $-C\leq m,n<M$.}
For each pair in this finite rectangle, put
\[
 J_{m,n}(f)=\int_{U^2}
 f(\varpi^m\xi,\varpi^n\eta)
 \kappa(\varpi^{m+n}\xi\eta)\Xi(\xi,\eta)
 \,d^\times\xi\,d^\times\eta.
\]
Then 
\[
 \mathcal T_{\mathrm I}(w)
 =\sum_{m=-C}^{M-1}\sum_{n=-C}^{M-1}
 J_{m,n}(f)q^mc_\omega^{-2m-n}z^{3m+n}
 =:R_f(z).
\]
Thus $R_f\in\C[z,z^{-1}]$. No asymptotic expansion is needed
in this region, and it contributes no denominator.

\ 

For use in Regions II--IV, set
\begin{equation}\label{eq:finite-geometric-ratios}
 X_a=\alpha_a c_\omega^{-2}qz^3,
 \qquad Y_a=\alpha_a c_\omega^{-1}z.
\end{equation}
Note 
\[
 q^m c_\omega^{-2m-n}z^{3m+n}\alpha_a^{m+n}
 =X_a^mY_a^n.
\]
For every integer $p\geq0$, write
\[
 G_{M,p}(Z)=\sum_{h=M}^{\infty}h^pZ^h
 =\left(Z\frac{d}{dZ}\right)^p\frac{Z^M}{1-Z}
 \qquad(|Z|<1).
\]
`Since $M\geq0$, induction on $p$ shows that
the differentiated expression is rational with polynomial numerator
and denominator dividing $(1-Z)^{p+1}$. As $\Re w\to+\infty$,
both $|X_a|$ and $|Y_a|$ tend to zero. For now we work far enough to the right so they have absolute value less than $1$. 

\medskip
\noindent\emph{Region II: $m\geq M$ and $-C\leq n<M$.}
Since $m\geq M\geq M_0$, the vector
$(\varpi^m\xi,0)$ belongs to $\varpi^{M_0}\mathcal O_K^2$.
Translation invariance therefore gives
\[
 f(\varpi^m\xi,\varpi^n\eta)=f(0,\varpi^n\eta).
\]
Also $m+n\geq M-C\geq\ell_0$. Hence the exact Whittaker
expansion \eqref{eq:finite-kirillov-germ}, obtained above from
\cite[Lemma~2.16.1]{JL} and also supplied by
\cite[Theorem~2.1]{Matringe}, gives
\[
 \kappa(\varpi^{m+n}\xi\eta)
 =\sum_{a=1}^r\sum_{j=0}^{d_a}
 c_{a,j}(\xi\eta)\alpha_a^{m+n}(m+n)^j.
\]
Define
\[
 A_{a,j,n}(f)=\int_{U^2}
 f(0,\varpi^n\eta)c_{a,j}(\xi\eta)\Xi(\xi,\eta)
 \,d^\times\xi\,d^\times\eta.
\]
This coefficient is independent of $m$ and $w$, is linear in $f$,
and satisfies
$|A_{a,j,n}(f)|\leq\|f\|_\infty\|c_{a,j}\|_\infty$.
Substituting the expansion and using
\eqref{eq:finite-geometric-ratios}, we obtain
\[
 \begin{aligned}
 \mathcal T_{\mathrm{II}}(w)
 &=\sum_{a=1}^r\sum_{j=0}^{d_a}\sum_{n=-C}^{M-1}
 A_{a,j,n}(f)Y_a^n
 \sum_{m=M}^{\infty}X_a^m(m+n)^j.
 \end{aligned}
\]
For each fixed $n$, the binomial theorem gives
\[
 \begin{aligned}
 \sum_{m=M}^{\infty}X_a^m(m+n)^j
 &=\sum_{p=0}^j\binom jp n^{j-p}
       \sum_{m=M}^{\infty}m^pX_a^m\\
 &=\sum_{p=0}^j\binom jp n^{j-p}G_{M,p}(X_a).
 \end{aligned}
\]
Therefore
\[
 \begin{aligned}
 \mathcal T_{\mathrm{II}}(w)
 &=\sum_{a=1}^r\sum_{j=0}^{d_a}\sum_{n=-C}^{M-1}
 A_{a,j,n}(f)Y_a^n\\
 &\qquad\times\sum_{p=0}^j
       \binom jp n^{j-p}G_{M,p}(X_a).
 \end{aligned}
\]
For the terms belonging to a fixed $a$, the denominator divides
$(1-X_a)^{d_a+1}$. The finitely many factors $Y_a^n$ are Laurent
monomials in $z$ and introduce no poles for $z\ne0$.

\medskip
\noindent\emph{Region III: $-C\leq m<M$ and $n\geq M$.}
Now $(0,\varpi^n\eta)\in\varpi^{M_0}\mathcal O_K^2$, so
\[
 f(\varpi^m\xi,\varpi^n\eta)=f(\varpi^m\xi,0).
\]
The inequality $m+n\geq M-C\geq\ell_0$ again permits the
exact expansion \eqref{eq:finite-kirillov-germ}. Set
\[
 B_{a,j,m}(f)=\int_{U^2}
 f(\varpi^m\xi,0)c_{a,j}(\xi\eta)\Xi(\xi,\eta)
 \,d^\times\xi\,d^\times\eta.
\]
This time the coefficient is independent of $n$ and $w$; it is
linear in $f$ and bounded in absolute value by
$\|f\|_\infty\|c_{a,j}\|_\infty$. Thus
\[
 \mathcal T_{\mathrm{III}}(w)
 =\sum_{a=1}^r\sum_{j=0}^{d_a}\sum_{m=-C}^{M-1}
 B_{a,j,m}(f)X_a^m
 \sum_{n=M}^{\infty}Y_a^n(m+n)^j.
\]
Expanding in powers of the unbounded variable $n$ yields
\[
 \begin{aligned}
 \sum_{n=M}^{\infty}Y_a^n(m+n)^j
 &=\sum_{p=0}^j\binom jp m^{j-p}
       \sum_{n=M}^{\infty}n^pY_a^n\\
 &=\sum_{p=0}^j\binom jp m^{j-p}G_{M,p}(Y_a).
 \end{aligned}
\]
Consequently
\[
 \begin{aligned}
 \mathcal T_{\mathrm{III}}(w)
 &=\sum_{a=1}^r\sum_{j=0}^{d_a}\sum_{m=-C}^{M-1}
 B_{a,j,m}(f)X_a^m\\
 &\qquad\times\sum_{p=0}^j
       \binom jp m^{j-p}G_{M,p}(Y_a).
 \end{aligned}
\]
For each fixed $a$, the denominator divides $(1-Y_a)^{d_a+1}$.
The finite factors $X_a^m$ are Laurent monomials in $z$.

\medskip
\noindent\emph{Region IV: $m,n\geq M$.}
Both coordinates lie in $\varpi^{M_0}\mathcal O_K$, and hence
\[
 f(\varpi^m\xi,\varpi^n\eta)=f(0,0).
\]
Also $m+n\geq2M\geq\ell_0$, so the same Whittaker expansion
applies. Define the constant
\[
 H_{a,j}=\int_{U^2}
 c_{a,j}(\xi\eta)\Xi(\xi,\eta)
 \,d^\times\xi\,d^\times\eta.
\]
It is independent of $m,n,w,f$ and satisfies
$|H_{a,j}|\leq\|c_{a,j}\|_\infty$. We obtain
\[
 \mathcal T_{\mathrm{IV}}(w)
 =f(0,0)\sum_{a=1}^r\sum_{j=0}^{d_a}H_{a,j}
   \sum_{m,n\geq M}X_a^mY_a^n(m+n)^j.
\]
For $|X_a|,|Y_a|<1$, the last double series is absolutely
convergent. Indeed $(m+n)^j$ is a finite sum of products
$\binom jp m^pn^{j-p}$, each summable against
$|X_a|^m|Y_a|^n$. Thus the binomial theorem and Fubini's theorem give
\[
 \begin{aligned}
 \sum_{m,n\geq M}X_a^mY_a^n(m+n)^j
 &=\sum_{p=0}^j\binom jp
   \left(\sum_{m=M}^{\infty}m^pX_a^m\right)
   \left(\sum_{n=M}^{\infty}n^{j-p}Y_a^n\right)\\
 &=\sum_{p=0}^j\binom jp
       G_{M,p}(X_a)G_{M,j-p}(Y_a).
 \end{aligned}
\]
Therefore
\[
 \mathcal T_{\mathrm{IV}}(w)
 =f(0,0)\sum_{a=1}^r\sum_{j=0}^{d_a}H_{a,j}
   \sum_{p=0}^j\binom jp G_{M,p}(X_a)G_{M,j-p}(Y_a).
\]
The term indexed by $a,j,p$ has denominator dividing
$(1-X_a)^{p+1}(1-Y_a)^{j-p+1}$. Since $j\leq d_a$, all terms
belonging to $a$ are cleared by
$(1-X_a)^{d_a+1}(1-Y_a)^{d_a+1}$.

\medskip
Combining Regions I--IV, define
\begin{equation}\label{eq:finite-denominator-polynomial}
 \mathscr D_W(z)=\prod_{a=1}^r
 (1-\alpha_a c_\omega^{-2}qz^3)^{d_a+1}
 (1-\alpha_a c_\omega^{-1}z)^{d_a+1}.
\end{equation}
The calculations above prove that
\begin{equation}\label{eq:finite-cleared-polynomial}
 \mathscr D_W(q^{-w})T_{W,f}(w)=P_f(q^{-w}),
 \qquad P_f\in\C[z,z^{-1}].
\end{equation}

The polynomial $\mathscr D_W$ has constant term $1$. Its factors
and their exponents depend only on the fixed Whittaker data, and
not on $f$, $C$, $M_0$, or $M$. 

Equation \eqref{eq:finite-cleared-polynomial} gives the required
meromorphic continuation, rational in $q^{-w}$. Different choices
of $C,M_0,M$ give the same continuation because their rational
expressions agree with the original integral in a right half-plane.
Every Laurent monomial in $q^{-w}$ is entire as a function of $w$.
Consequently
\[
 p(w)=\mathscr D_W(q^{-w})
\]
is a nonzero entire function, independent of $f$, for which
$p(w)T_{W,f}(w)$ is entire. 

It remains to check the uniform distribution bounds. Fix
$C,M_0\geq0$. The space $\mathcal S_{C,M_0}$ from
Definition~\ref{def:meromorphic-distribution} has the finite basis
\[
 e_b=\mathbf1_{b+\varpi^{M_0}\mathcal O_K^2},
 \qquad b\in\varpi^{-C}\mathcal O_K^2/
                       \varpi^{M_0}\mathcal O_K^2.
\]
Choose one representative for each coset. Translation invariance
gives $f=\sum_b f(b)e_b$. In the domain of absolute convergence we have 
\[
 p(w)T_{W,f}(w)=\sum_b f(b)P_{e_b}(q^{-w}).
\]
By analytic continuation the same identity holds for all $w$. 
Thus, for every compact set $Q\subset\C$,
\[
 \sup_{w\in Q}|p(w)T_{W,f}(w)|
 \leq
 \left(\sum_b\sup_{w\in Q}|P_{e_b}(q^{-w})|\right)
 \|f\|_\infty.
\]
The constant in parentheses is finite because the basis is finite
and every $P_{e_b}(q^{-w})$ is entire. These are precisely the
bounds required in Definition~\ref{def:meromorphic-distribution}.
They also imply continuity on the usual locally convex inductive
limit $\cS(K^2)$. Together with the absolute convergence and
holomorphy established above, this completes the proof of the  nonarchimedean portion of Proposition~\ref{prop:local}.

\subsection{Proof of Proposition \ref{prop:local} in the archimedean case}\label{subsec:sector}
We start with a lemma where we apply a method of Gel'fand--Shilov \cite[Chapter~I, Section~3.2]{GS}.   A Schwartz function on the closed positive quadrant means a smooth function, with all derivatives extending to the boundary, for which every seminorm
\[
 q_{L,M}(H)=\max_{i+j\leq M}\sup_{r,s\geq0}
       (1+r+s)^L|\partial_r^i\partial_s^jH(r,s)|
\]
is finite. The analogous convention is used on the closed half-line. 

\begin{lemma}\label{lem:sector}
Let $h\in C^\infty(\R_{>0})$ satisfy $h(t)=O(t^{-B})$ for some $B\geq0$, and suppose that $h$ decreases faster than every power of $t$ as $t$ approaches infinity. Suppose also that its Mellin transform
\[
 \mathcal M_h(z)=\int_0^\infty h(t)t^z\,\frac{dt}{t}
\]
has a meromorphic continuation to $\C$. Fix real numbers $e, \lambda_1,\lambda_2$ with $e > 0$, and put
\begin{equation}\label{eq:ab}
 a(w)=e(3w-1)+i\lambda_1,\qquad b(w)=ew+i\lambda_2.
\end{equation}
For a Schwartz function $H$ on the closed positive quadrant, the integral
\begin{equation}\label{eq:sector-integral}
 B_H(w)=\int_0^\infty\int_0^\infty
       H(r,s)h(rs)r^{a(w)}s^{b(w)}\,\frac{dr}{r}\frac{ds}{s}
\end{equation}
is convergent for $\Re w$ sufficiently large. 
Locally in \(w\), there is a holomorphic function \(p(w)\), not identically zero and independent of \(H\), such that \(p(w)B_H(w)\) is holomorphic and
\begin{equation}\label{eq:sector-uniform-bound}
\sup_{w\in Q}|p(w)B_H(w)|\le C_Qq_Q(H)
\end{equation}
on each compact parameter set \(Q\).

\end{lemma}

\begin{proof} The proof has several steps. 

\ 

Step 1. We start by proving the absolute convergence. By the assumption on $h$, 
\begin{equation}\label{(3)"}
|h(t)|\le C_h t^{-B}\qquad(t>0)
\end{equation}
for some constant $C_h$.  Consequently,
\[
|H(r,s)h(rs)r^{a(w)}s^{b(w)}|
\le
C_h|H(r,s)|
r^{\Re a(w)-B}s^{\Re b(w)-B}.
\]
For any integer \(L\ge0\),
\begin{equation}\label{(4)"}
|H(r,s)|
\le q_{2L,0}(H)(1+r)^{-L}(1+s)^{-L},
\end{equation}
because
\[
(1+r)(1+s)\le(1+r+s)^2.
\]
If
\begin{equation}\label{(5)"}
\Re a(w)>B,\qquad \Re b(w)>B,
\end{equation}
choose \(L\) larger than both positive exponents
\(\Re a(w)-B\) and \(\Re b(w)-B\). By Equations \eqref{(3)"} and \eqref{(4)"}, $|B_H(w)|$ is bounded by a product of two convergent one-dimensional integrals. The same argument is uniform on compact subsets of the region \eqref{(5)"}. This also implies $B_H(w)$ defines a holomorphic function in this nonempty right half-plane.

\ 

Step 2. Let us recall the method of Gel'fand--Shilov \cite[Chapter~I, Section~3.2]{GS}. 
Let \(g\) be Schwartz on \([0,\infty)\), and write
\[
g_n=\frac{g^{(n)}(0)}{n!}.
\]
For \(\Re c>0\), split its Mellin integral at \(1\), subtract the first \(J\) Taylor terms on \((0,1)\), and integrate those terms explicitly:
\begin{equation}\label{(6)"}
\begin{aligned}
\mathcal M(g;c)
={}&\int_1^\infty g(R)R^c\,\frac{dR}{R}\\
&+\int_0^1
\left(g(R)-\sum_{n=0}^{J-1}g_nR^n\right)
R^c\,\frac{dR}{R}
+\sum_{n=0}^{J-1}\frac{g_n}{c+n}.
\end{aligned}
\end{equation}

Taylor’s theorem gives
\begin{equation}\label{(7)"}
\left|g(R)-\sum_{n=0}^{J-1}g_nR^n\right|
\le
\frac{R^J}{J!}\sup_{0\le u\le1}|g^{(J)}(u)|.
\end{equation}
Therefore the second integral in \eqref{(6)"} is holomorphic for
\(\Re c>-J\).

More precisely, on a compact set where
\[
\Re c+J\ge\delta>0,
\]
its \(m\)-th parameter derivative is bounded by a seminorm of \(g\) times
\begin{equation}\label{(8)"}
\int_0^1 R^{\delta-1}|\log R|^m\,dR
=\frac{m!}{\delta^{m+1}}.
\end{equation}
The first integral in \eqref{(6)"} is entire, with analogous seminorm bounds obtained from rapid decrease at infinity.

Thus \eqref{(6)"} gives meromorphic continuation to \(\Re c>-J\), with possible simple poles at
\[
c=0,-1,\ldots,-J+1.
\]
On this half-plane,
\[
\prod_{n=0}^{J-1}(c+n)\,\mathcal M(g;c)
\]
is holomorphic and locally bounded by a continuous seminorm of \(g\). Increasing \(J\) gives compatible continuations because the formulas agree in \(\Re c>0\).

\ 

Step 3.  Choose \(\kappa\in C^\infty([0,\infty))\) such that
\[
\kappa(t)=1\quad(0\le t\le\delta),
\qquad 
0 \leq \kappa(t) \leq 1 \quad (\forall t \geq 0), 
\qquad
\kappa(t)=0\quad(t\ge1),
\]
for some \(0<\delta<1\). Set 

\begin{equation*}
 B_H^{\mathrm{far}}(w)=\int_0^\infty\int_0^\infty
       H(r,s)h(rs)(1-\kappa(rs))r^{a(w)}s^{b(w)}\,\frac{dr}{r}\frac{ds}{s}
\end{equation*}
and
\begin{equation*}
 B_H^{\mathrm{near}}(w)=\int_0^\infty\int_0^\infty
       H(r,s)h(rs)\kappa(rs)r^{a(w)}s^{b(w)}\,\frac{dr}{r}\frac{ds}{s}
\end{equation*}
Then
\[
B_H(w)=B_H^{\mathrm{far}}(w)+B_H^{\mathrm{near}}(w).
\]

On the support of \(1-\kappa(rs)\),
\[
rs\ge\delta.
\]
Hence
\begin{equation}\label{(9)"}
r^{-1}\le\delta^{-1}s,\qquad
s^{-1}\le\delta^{-1}r,\qquad
\frac1{rs}\le\delta^{-1}.
\end{equation}

For \(w\) in a fixed compact set, all powers
\[
|r^{a(w)}s^{b(w)}|
\]
are bounded by a fixed polynomial in \(1+r+s\), using \eqref{(9)"} for negative powers. The same holds after inserting any fixed powers of \(\log r,\log s\). Also, \(h\) is bounded on \([\delta,\infty)\). As $H$ is a Schwartz function, it follows immediately that \(B_H^{\mathrm{far}}\) is entire, with locally uniform Schwartz semi-norm bounds.

\ 

Step 4. In the remainder of the proof we handle $B_H^{\mathrm{near}}(w)$.  Split the positive quadrant into
\[
r\le s\qquad\text{and}\qquad s\le r.
\]
Their intersection, the diagonal, has measure zero.

\ 

Consider the first sector. Put
\[
t=rs,\qquad R=s.
\]
Then
\[
r=t/R,\qquad R\ge\sqrt t,
\]
and
\[
\frac{dr}{r}\frac{ds}{s}
=
\frac{dt}{t}\frac{dR}{R}.
\]
Moreover,
\[
r^{a(w)}s^{b(w)}
=t^{a(w)}R^{b(w)-a(w)}.
\]
Thus this sector contributes
\begin{equation}\label{(10)"}
B_H^{(1)}(w)
=
\int_0^1\kappa(t)h(t)K_H(w,t)\,\frac{dt}{t},
\end{equation}
where
\begin{equation}\label{(11)"}
K_H(w,t)
=
t^{a(w)}
\int_{\sqrt t}^{\infty}
H(t/R,R)R^{b(w)-a(w)}\,\frac{dR}{R}.
\end{equation}

Fix a bounded open parameter neighborhood \(\Omega\), with compact closure. All Taylor orders and estimates below will be chosen uniformly on \(\overline\Omega\), independently of \(H\).

\ 

Step 5.  Define
\[
H_m(R)=\frac{\partial_1^mH(0,R)}{m!}.
\]
Taylor’s theorem in the first coordinate gives
\begin{equation}\label{(12)"}
H(t/R,R)
=
\sum_{m=0}^{N-1}(t/R)^mH_m(R)
+(t/R)^N U_N(t/R,R),
\end{equation}
where
\begin{equation}\label{(13)"}
U_N(r,R)
=
\frac1{(N-1)!}
\int_0^1(1-\theta)^{N-1}
\partial_1^NH(\theta r,R)\,d\theta.
\end{equation}

For every \(L\),
\begin{equation}\label{(14)"}
|U_N(r,R)|
\le
\frac1{N!}\,q_{L,N}(H)(1+R)^{-L}.
\end{equation}
Indeed, \(1+\theta r+R\ge1+R\). This estimate is valid uniformly when \(R\to\infty\) and \(r=t/R\to0\). 

Substituting \eqref{(12)"} into \eqref{(11)"}, put
\[
c_m(w)=b(w)-a(w)-m.
\]
Then
\begin{equation}\label{(15)"}
K_H(w,t)
=
\sum_{m=0}^{N-1}
t^{a(w)+m}
\int_{\sqrt t}^{\infty}H_m(R)R^{c_m(w)}\,\frac{dR}{R}
+\mathcal E_N(w,t),
\end{equation}
where
\begin{equation}\label{(16)"}
\mathcal E_N(w,t)
=
t^{a(w)+N}
\int_{\sqrt t}^{\infty}
U_N(t/R,R)R^{b(w)-a(w)-N}\,\frac{dR}{R}.
\end{equation}

The maps \(H\mapsto H_m\) are continuous between the corresponding Schwartz spaces. Explicitly,
\begin{equation}\label{(17)"}
\sup_{R\ge0}(1+R)^L|H_m^{(j)}(R)|
\le \frac1{m!}q_{L,m+j}(H).
\end{equation}

\ 

Step 6.  Write
\[
\beta_{mn}
=
\frac{\partial_1^m\partial_2^nH(0,0)}{m!n!},
\qquad
G_{m,J}(R)=H_m(R)-\sum_{n=0}^{J-1}\beta_{mn}R^n.
\]
Choose \(J\) so large that
\begin{equation}\label{(18)"}
\Re c_m(w)+J>0, 
\qquad w\in\overline\Omega,\ 0\le m<N.
\end{equation}

For fixed \(t>0\), Equation \eqref{(6)"} gives
\begin{equation}\label{(19)"}
\begin{aligned}
\int_{\sqrt t}^{\infty}H_m(R)R^{c_m(w)}\,\frac{dR}{R}
={}&\mathcal M(H_m;c_m(w))\\
&-\sum_{n=0}^{J-1}
\frac{\beta_{mn}t^{(c_m(w)+n)/2}}{c_m(w)+n}\\
&-\int_0^{\sqrt t}G_{m,J}(R)R^{c_m(w)}\,\frac{dR}{R}.
\end{aligned}
\end{equation}
To justify the identity, first take \(\Re c_m>0\), where it follows by subtracting the integral over \((0,\sqrt t)\) from the ordinary Mellin integral. Formula \eqref{(6)"} then continues both sides to \(\Re c_m>-J\).

For fixed \(t>0\), the left side of \eqref{(19)"} is entire in \(c_m\). Therefore any apparent poles on the right cancel in the complete expression. 

Define
\[
\mathcal V_{m,J}(w,t)
=
t^{a(w)+m}
\int_0^{\sqrt t}G_{m,J}(R)R^{c_m(w)}\,\frac{dR}{R}.
\]
Combining \eqref{(15)"} and \eqref{(19)"} gives
\begin{equation}\label{(20)"}
\begin{aligned}
K_H(w,t)
={}&
\sum_{m=0}^{N-1}
\mathcal M(H_m;c_m(w))\,t^{a(w)+m}\\
&-
\sum_{\substack{0\le m<N\\0\le n<J}}
\frac{\beta_{mn}}{c_m(w)+n}
t^{(a(w)+b(w)+m+n)/2}\\
&+\mathcal R_{N,J}(w,t),
\end{aligned}
\end{equation}
where
\begin{equation}\label{(21)"}
\mathcal R_{N,J}
=
\mathcal E_N-\sum_{m=0}^{N-1}\mathcal V_{m,J}.
\end{equation}

\ 

Step 7.  First consider \(\mathcal E_N\). By \eqref{(14)"},
\begin{equation}\label{(22)"}
|\mathcal E_N(w,t)|
\le
Cq_{L,N}(H)t^{\Re a(w)+N}
\int_{\sqrt t}^{\infty}
R^{\Re(b(w)-a(w))-N}(1+R)^{-L}\,\frac{dR}{R}.
\end{equation}
Set
\[
\nu(w)=\Re(b(w)-a(w))-N.
\]

Choose \(L\) sufficiently large that
\[
L-\sup_{w\in\overline\Omega}\nu(w)\ge1.
\]
Then
\begin{equation}\label{(23)"}
\int_1^\infty R^{\nu(w)}(1+R)^{-L}\,\frac{dR}{R}
\le1.
\end{equation}

For the interval \([\sqrt t,1]\), every value of \(R^{\nu(w)}\) is bounded by the larger endpoint value. Thus
\begin{equation}\label{(24)"}
\begin{aligned}
\int_{\sqrt t}^{1}R^{\nu(w)}\,\frac{dR}{R}
&\le
\frac12|\log t|\max\{1,t^{\nu(w)/2}\}\\
&\le
\frac12|\log t|\bigl(1+t^{\nu(w)/2}\bigr).
\end{aligned}
\end{equation}
This estimate remains uniform when \(\nu(w)=0\); there is no division by \(\nu(w)\).

Equations \eqref{(22)"}–\eqref{(24)"} imply
\begin{equation}\label{(25)"}
|\mathcal E_N(w,t)|
\le
Cq(H)(1+|\log t|)
\left(
t^{\Re a(w)+N}
+
t^{(\Re(a(w)+b(w))+N)/2}
\right).
\end{equation}

Next, Taylor’s theorem for \(H_m\) gives, for \(0\le R\le1\),
\[
|G_{m,J}(R)|
\le
\frac{q_{0,m+J}(H)}{m!J!}R^J.
\]
If \eqref{(18)"} holds with a uniform lower bound
\[
\Re c_m(w)+J\ge\delta_1>0,
\]
then
\begin{equation}\label{(26)"}
\begin{aligned}
|\mathcal V_{m,J}(w,t)|
&\le
Cq(H)t^{\Re a(w)+m}
\int_0^{\sqrt t}R^{\Re c_m(w)+J}\,\frac{dR}{R}\\
&\le
Cq(H)t^{(\Re(a(w)+b(w))+m+J)/2}.
\end{aligned}
\end{equation}

The choices can now be made in the required order.

First choose \(N\ge1\) so large that, throughout \(\overline\Omega\),
\begin{equation}\label{(27)"}
\Re a(w)+N\ge B+1,
\qquad
\frac{\Re(a(w)+b(w))+N}{2}\ge B+1.
\end{equation}
This is possible because \(\Omega\) is bounded.

Having fixed \(N\), choose \(J\ge1\) so large that, for all \(m<N\),
\begin{equation}\label{(28)"}
\Re c_m(w)+J\ge1,
\qquad
\frac{\Re(a(w)+b(w))+m+J}{2}\ge B+1
\end{equation}
throughout \(\overline\Omega\).

Finally choose \(L\) as required in \eqref{(23)"}. A single sufficiently large Schwartz seminorm dominates the finitely many seminorms occurring above. Equations \eqref{(21)"}, \eqref{(25)"}, and \eqref{(26)"} therefore give
\begin{equation}\label{(29)"}
|\mathcal R_{N,J}(w,t)|
\le Cq(H)t^{B+1}(1+|\log t|),
\qquad 0<t<1,\quad w\in\overline\Omega.
\end{equation}

For fixed \(t>0\), the remainder is holomorphic in \(w\): its explicit integral expressions converge locally uniformly under the chosen inequalities. Multiplication by \(h(t)\) gives
\begin{equation}\label{(30)"}
\left|
\kappa(t)h(t)\mathcal R_{N,J}(w,t)\frac1t
\right|
\le
Cq(H)(1+|\log t|).
\end{equation}
The right side is integrable over \((0,1)\). Consequently,
\begin{equation}\label{(31)"}
\mathcal H_H(w)
=
\int_0^1\kappa(t)h(t)\mathcal R_{N,J}(w,t)\,\frac{dt}{t}
\end{equation}
is holomorphic on \(\Omega\), and
\[
\sup_{w\in Q}|\mathcal H_H(w)|\le C_Qq_Q(H)
\]
for every compact \(Q\subset\Omega\).

\ 

Step 8.  Define initially for \(\Re z>B\)
\[
\mathfrak m_h(z)
=
\int_0^1\kappa(t)h(t)t^z\,\frac{dt}{t}.
\]
The difference
\begin{equation}\label{(32)"}
\mathcal M_h(z)-\mathfrak m_h(z)
=
\int_0^\infty(1-\kappa(t))h(t)t^z\,\frac{dt}{t}
\end{equation}
is entire. Its integrand is supported away from zero, and rapid decay of \(h\) gives convergence, including all parameter derivatives, at infinity. Thus \(\mathfrak m_h\) is meromorphic on \(\mathbb C\).

Integrating the finite terms in \eqref{(20)"} yields
\begin{equation}\label{(33)"}
\begin{aligned}
B_H^{(1)}(w)
={}&
\sum_{m=0}^{N-1}
\mathcal M(H_m;c_m(w))
\,\mathfrak m_h(a(w)+m)\\
&-
\sum_{\substack{0\le m<N\\0\le n<J}}
\frac{\beta_{mn}}{c_m(w)+n}
\,\mathfrak m_h\!\left(
\frac{a(w)+b(w)+m+n}{2}
\right)\\
&+\mathcal H_H(w).
\end{aligned}
\end{equation}
Where \eqref{(5)"} holds, all the finite \(t\)-integrals used in this calculation converge: their exponents have real part greater than \(B\). Away from the finitely many apparent coefficient poles, equation \eqref{(33)"} therefore agrees with the original sector integral.

\ 

Step 9. We can now finish the proof.  By \eqref{(28)"}, formula \eqref{(6)"} for \(\mathcal M(H_m;c_m(w))\) is valid throughout \(\Omega\). Its possible poles there are cleared by
\begin{equation}\label{(35)"}
P_{\mathrm{Taylor}}(w)
=
\prod_{m=0}^{N-1}\prod_{n=0}^{J-1}
(c_m(w)+n).
\end{equation}
The same polynomial clears every explicit denominator
\(1/(c_m(w)+n)\) in \eqref{(33)"}.

Moreover, after multiplying by \eqref{(35)"}, the ordinary Mellin factors are locally bounded by Schwartz seminorms of \(H_m\), hence by seminorms of \(H\) through \eqref{(17)"}. The coefficients \(\beta_{mn}\) satisfy
\begin{equation}\label{(36)"}
|\beta_{mn}|
\le\frac{q_{0,m+n}(H)}{m!n!}.
\end{equation}

There remain only finitely many fixed scalar meromorphic functions:
\begin{equation}\label{(37)"}
\mathfrak m_h(a(w)+m),
\qquad
\mathfrak m_h\!\left(
\frac{a(w)+b(w)+m+n}{2}
\right).
\end{equation}
They depend on \(h,e,N,J\), but not on \(H\). On a compact neighborhood of the parameter set under consideration, these functions have only finitely many poles. Choose a polynomial \(P_h(w)\) whose zeros cancel all of them, with sufficient multiplicities.

Then
\begin{equation}\label{(38)"}
p_1(w)=P_{\mathrm{Taylor}}(w)P_h(w)
\end{equation}
is independent of \(H\), is not identically zero, and makes the right side of \eqref{(33)"} holomorphic.

Combining the bounds from \eqref{(6)"}, \eqref{(17)"}, \eqref{(29)"}, and \eqref{(36)"}, and taking maxima of the finitely many fixed scalar factors, gives
\begin{equation}\label{(39)"}
\sup_{w\in Q}|p_1(w)B_H^{(1)}(w)|
\le C_Qq_Q(H).
\end{equation}
A finite sum of the seminorms appearing in this calculation is a continuous Schwartz seminorm, so \eqref{(39)"} has exactly the form required in the lemma. The proof for the sector $s \leq r$ is identical. 
\end{proof}

\begin{proof}[Proof of Proposition \ref{prop:local} in the archimedean setting]
Let $K=\R$ or $\C$, and $\Omega_\R = \{\pm 1\}$ and $\Omega_\C = S^1$. 

We first establish the absolute convergence and holomorphy in the domain \eqref{eq:localtube}. The argument is very similar to the nonarchimdean case.  On a compact subset of \eqref{eq:localtube}, choose $\epsilon>0$ smaller than both strict margins. Where $|xy|_K\leq1$, \eqref{eq:W-small} bounds the integrand by a constant times
\[
 |f(x,y)|\,
 |x|_K^{\Re u_1+1/2-\theta_K-\epsilon}
 |y|_K^{\Re u_2+1/2-\theta_K-\epsilon}.
\]
Both exponents are positive. This is integrable near either axis and at their intersection; away from the axes compact support or Schwartz decay absorbs the positive weights. At an archimedean place the far-product argument in Lemma~\ref{lem:sector} applies directly. Logarithmic factors are absorbed by shrinking the margins, so the integrals and their parameter derivatives converge normally. This proves local holomorphy with continuous test-function bounds.

Compact-finiteness implies a finite angular decomposition
\begin{equation}\label{eq:angular}
 W\bigl(a(1,r\omega)\bigr)=\sum_{\chi\in\mathcal X}\chi(\omega)h_\chi(r), \quad r >0, \omega \in \Omega_K
\end{equation}
where $\mathcal X$ is a finite subset of the characters of $\Omega_K$. Each $h_\chi$ is smooth, has at most power growth at zero, and decreases rapidly at infinity, by the bounds in Section~\ref{subsec:whittaker-facts}.

The projection onto the character $\chi$ in \eqref{eq:angular} is
\[
 h_\chi(r)=\int_{\Omega_K}W(a(1,r\omega))\overline{\chi(\omega)}\,d\omega.
\]
Extend $\chi$ to
$K^\times$ by $\chi(r\omega)=\chi(\omega)$.  Since
$|r\omega|_K=r^{e_K}$ and $d^\times t=e_K\,dr/r\,d\omega$, absolute
convergence for $\Re z$ sufficiently large gives
\begin{align*}
 \int_0^\infty h_\chi(r)r^z\,\frac{dr}{r}
 &=\int_0^\infty\int_{\Omega_K}
       W(a(1,r\omega))\overline{\chi(\omega)}r^z
                   \,d\omega\,\frac{dr}{r}\\
 &=\frac1{e_K}\int_{K^\times}
       W(a(1,t))\overline{\chi(t)}|t|_K^{z/e_K}\,d^\times t.
\end{align*}
Apply \eqref{eq:scalar-local-mellin} with
$\xi=\overline\chi$ and with its parameter equal to
$\frac12+\frac{z}{e_K}$.  Consequently
\[
 \frac{\displaystyle\int_0^\infty h_\chi(r)r^z\,dr/r}
 {L(\frac12+\frac{z}{e_K},\pi_K\otimes\overline\chi)}
\]
extends to an entire function of $z$.  This supplies exactly the scalar
meromorphic Mellin transform used in Lemma~\ref{lem:sector}. 

Write the local unitary character of the representation in polar coordinates as
\[
 \omega_K(r\zeta)=r^{i\lambda}\omega_K(\zeta),
 \qquad r>0,\quad\zeta\in\Omega_K,\quad\lambda\in\R.
\]
For $f\in\cS(K^2)$ put
\[
 H_\chi(r,s)=\int_{\Omega_K^2}
 f(r\zeta,s\nu)\chi(\zeta\nu)
 \omega_K(\zeta)^{-2}\omega_K(\nu)^{-1}\,d\zeta\,d\nu.
\]
This is smooth on the closed quadrant and rapidly decreasing with all derivatives. Every one of its Schwartz seminorms is bounded by a Schwartz seminorm of $f$.

With $e=e_K$, $\lambda_1=-2\lambda$ and $\lambda_2=-\lambda$,
the exact identity is
\[
 T_{W,f}(w)=e_K^2\sum_{\chi\in\mathcal X}
 \int_0^\infty\!\int_0^\infty
 H_\chi(r,s)h_\chi(rs)
 r^{e_K(3w-1)-2i\lambda}s^{e_Kw-i\lambda}
 \frac{dr}{r}\frac{ds}{s}.
\]
Apply Lemma~\ref{lem:sector} to each term and multiply the
finitely many clearing factors. The seminorm bounds for the
angular averages prove \eqref{eq:seminorm-bound}.
\end{proof}

\section{The global Euler factorization}\label{sec:factorization}

\subsection{A finite set of places}
Enlarge $S_0$ so that outside it the additive character has conductor $\OO_v$, the central character $\omega_v$ is unramified, the test functions are standard, and the representation is spherical. After taking finite sums as necessary, write
\[
 F=F_{S_0}\otimes\bigotimes_{v\notin S_0}\one_{\OO_v^2},
 \qquad
 W=W_{S_0}\otimes\bigotimes_{v\notin S_0}W_v^\circ,
 \qquad W_v^\circ(1)=1.
\]
The $S_0$-adic test function $F_{S_0}$ is not required to be decomposable. We also express $W_{S_0}$ as a finite sum of decomposable Whittaker vectors. We may refine the index $j$ in \eqref{eq:uncut} to include all these finite sums.

The following result extends \cite[Theorems~5--6]{Simon} to meromorphic families of distributions. 

\begin{lemma}
\label{lem:tensor-meromorphic}
Let $S$ be a finite set of local fields and, for $v\in S$, let
$T_v(\eta)$ be a meromorphic family of distributions on
$\cS(k_v^{d_v})$, where $\eta$ belongs to an open subset of
$\C^d$. Suppose that locally the poles are cleared by a nonzero
holomorphic scalar independent of the test function, and that the
cleared families satisfy uniform Schwartz-seminorm bounds (or,
at finite places, bounds on each fixed support and translation level).
Then
\[
 F\longmapsto\left(\bigotimes_{v\in S}T_v(\eta)\right)(F)
\]
is a meromorphic distribution on
$\cS(\prod_{v\in S}k_v^{d_v})$, with the same type of local bounds.
This applies to arbitrary, possibly nondecomposable, test functions.
\end{lemma}

\begin{proof} The proof has several steps. 

\ 

Step 1.  Fix a point $\eta_0$ in the parameter domain. Since $S$ is finite,
we may choose a common connected polydisc $\Omega$ about $\eta_0$
and holomorphic functions $p_v$ on $\Omega$, none identically zero,
such that
\[
 A_v(\eta)=p_v(\eta)T_v(\eta),\qquad v\in S,
\]
are holomorphic distribution families on $\Omega$. By hypothesis, their values on each fixed test function are holomorphic, and their distribution bounds are uniform on compact subsets of $\Omega$. We will construct a holomorphic distribution family $A(\eta)$ satisfying
\[
 A(\eta)\left(\bigotimes_{v\in S}f_v\right)
 =\prod_{v\in S}A_v(\eta)(f_v).
\]
The required local clearing factor will be
$p(\eta)=\prod_{v\in S}p_v(\eta)$.

\ 

Step 2. Write $S=S_\infty\sqcup S_f$, according as $k_v$ is archimedean
or nonarchimedean. For $v\in S_\infty$, identify $k_v^{d_v}$
with $\R^{n_v}$, where $n_v=d_v$ for $k_v=\R$ and $n_v=2d_v$
for $k_v=\C$, and put
\[
 N=\sum_{v\in S_\infty}n_v.
\]

Fix support and translation level in the non-archimedean components. More
precisely, for each $v\in S_f$, choose integers $C_v,M_v\geq0$
and set
\[
 K_v=\varpi_v^{-C_v}\mathcal O_v^{d_v},\qquad
 L_v=\varpi_v^{M_v}\mathcal O_v^{d_v}.
\]
Let $B_v$ be a set of representatives for the finite quotient
$K_v/L_v$, and write
\[
 B=\prod_{v\in S_f}B_v,\qquad
 e_{v,b_v}=\mathbf1_{b_v+L_v},\qquad
 e_b=\bigotimes_{v\in S_f}e_{v,b_v}.
\]
Every Schwartz--Bruhat function $F$ supported in
$\R^N\times\prod_{v\in S_f}K_v$ and invariant under translation
by $\prod_{v\in S_f}L_v$ has the unique expansion
\[
 F(x_\infty,x_f)=\sum_{b\in B}F_b(x_\infty)e_b(x_f),
 \qquad F_b(x_\infty)=F(x_\infty,b),
\]
with $F_b\in\cS(\R^N)$. This fixed-support, fixed-level space
is a finite direct sum of copies of $\cS(\R^N)$, with its usual
Fr\'echet topology. Every Schwartz--Bruhat function belongs to
such a space. If $S_f$ is empty, $B$ consists of one element and
the finite-variable factor is omitted.

\ 

Step 3. If $N=0$, then every test function on a
fixed support and level is $F=\sum_{b\in B}F(b)e_b$, and one sets
\[
 A(\eta)(F)=\sum_{b\in B}F(b)
                 \prod_{v\in S_f}A_v(\eta)(e_{v,b_v}).
\]
This is a finite sum of holomorphic functions and satisfies
\[
 \sup_{\eta\in Q}|A(\eta)(F)|
 \leq |B|\prod_{v\in S_f}C_{v,Q,C_v,M_v}\,\|F\|_\infty.
\]
Compatibility and continuity follow as above. Empty products
and an empty set of places are interpreted in the usual way,
so the scalar case is included as well.

\ 

Step 4. Next we treat the case where $N \geq 1$.  We recall the Hermite estimates needed for the archimedean
variables. This is the Hermite-expansion method for Schwartz
functions and tempered distributions used by Simon \cite{Simon};
we give the estimates explicitly. Let
$h_\nu^{(n)}$, $\nu\in\mathbb N_0^n$, be the normalized
real-valued Hermite basis of $L^2(\R^n)$. Explicitly, 
\[
h_j(x)=\frac{H_j(x)e^{-x^2/2}}
{\bigl(2^j j!\sqrt{\pi}\bigr)^{1/2}},
\qquad j=0,1,2,\ldots,
\]
where
\[
H_j(x)=(-1)^j e^{x^2}\frac{d^j}{dx^j}e^{-x^2}.
\]
Then, if $\nu = (\nu_1, \dots, \nu_n)$ and $x = (x_1, \dots, x_n)$, we set 
$$
h_\nu^{(n)}(x) = \prod_{j=1}^n h_{\nu_j}(x_j). 
$$
For the standard basis
vector $\varepsilon_j$ of $\mathbb Z^n$, its recurrence relations are
\[
 \begin{aligned}
 x_jh_\nu^{(n)}
 &=\sqrt{\frac{\nu_j}{2}}h_{\nu-\varepsilon_j}^{(n)}
   +\sqrt{\frac{\nu_j+1}{2}}h_{\nu+\varepsilon_j}^{(n)},\\
 \partial_{x_j}h_\nu^{(n)}
 &=\sqrt{\frac{\nu_j}{2}}h_{\nu-\varepsilon_j}^{(n)}
   -\sqrt{\frac{\nu_j+1}{2}}h_{\nu+\varepsilon_j}^{(n)}.
 \end{aligned}
\]
 Terms with a negative index are interpreted as zero. These relations follow, for example, from the non-normalized identities (1.1.5) and (1.1.6) of \cite{Tha}. Repeated
application shows that, for any multi-indices $\alpha,\beta$,
\[
 \|x^\alpha\partial^\beta h_\nu^{(n)}\|_{L^2}
 \leq C_{\alpha,\beta,n}
          (1+|\nu|)^{(|\alpha|+|\beta|)/2}.
\]
Indeed, each multiplication or differentiation changes the index
by one and introduces a coefficient bounded by a constant times
$(1+|\nu|)^{1/2}$; the number of terms depends only on the
number of operations.

\ 

Step 5. Choose an integer $s>n/2$. Sobolev's inequality \cite[Theorem~9.17]{Folland}, applied to
$x^\alpha\partial^\beta h_\nu^{(n)}$, and the product rule give
\[
 \begin{aligned}
 \|x^\alpha\partial^\beta h_\nu^{(n)}\|_\infty
 &\leq C_{n,s}\sum_{|\gamma|\leq s}
       \|\partial^\gamma(x^\alpha\partial^\beta
                                      h_\nu^{(n)})\|_{L^2}\\
 &\leq C_{\alpha,\beta,n,s}
       (1+|\nu|)^{(|\alpha|+|\beta|+s)/2}.
 \end{aligned}
\]
Every continuous Schwartz seminorm is bounded by a constant times
a finite sum of seminorms
$g\mapsto\|x^\alpha\partial^\beta g\|_\infty$.
Consequently, for each continuous Schwartz seminorm $q$ there
exist constants $C_q>0$ and $m_q\geq0$ such that
\[
 q(h_\nu^{(n)})\leq C_q(1+|\nu|)^{m_q}.
\]

\ 

Step 6. We also need a bound for the Hermite coefficients of a test
function. Write $\boldsymbol\nu=(\nu_v)_{v\in S_\infty}$,
where $\nu_v\in\mathbb N_0^{n_v}$, and set
\[
 h_{\boldsymbol\nu}(x_\infty)
 =\prod_{v\in S_\infty}h_{\nu_v}^{(n_v)}(x_v),
 \qquad |\boldsymbol\nu|=\sum_{v\in S_\infty}|\nu_v|.
\]
These functions form the normalized Hermite basis of $L^2(\R^N)$.
The operator
\[
 \mathscr L_N=1-\Delta_{\R^N}+|x_\infty|^2
\]
satisfies
\[
 \mathscr L_Nh_{\boldsymbol\nu}
 =(1+N+2|\boldsymbol\nu|)h_{\boldsymbol\nu}.
\]
Define
\[
 c_{b,\boldsymbol\nu}(F)
 =\int_{\R^N}F_b(x_\infty)
               \overline{h_{\boldsymbol\nu}(x_\infty)}\,dx_\infty.
\]
For every integer $L\geq0$, integration by parts and
Cauchy--Schwarz give
\[
 \begin{aligned}
 |c_{b,\boldsymbol\nu}(F)|
 &= (1+N+2|\boldsymbol\nu|)^{-L}
    \left|\int_{\R^N}(\mathscr L_N^LF_b)
                        \overline{h_{\boldsymbol\nu}}\,dx_\infty\right|\\
 &\leq (1+|\boldsymbol\nu|)^{-L}
                \|\mathscr L_N^LF_b\|_{L^2}.
 \end{aligned}
\]
The map $g\mapsto\|\mathscr L_N^Lg\|_{L^2}$ is bounded by a 
Schwartz seminorm. Indeed, choose an integer \(M>N/2\). Then
\[
\|u\|_{L^2(\mathbb R^N)}
\le C_{N,M}\sup_{x\in\mathbb R^N}(1+|x|)^M|u(x)|,
\]
where
\[
C_{N,M}
=\left(\int_{\mathbb R^N}(1+|x|)^{-2M}\,dx\right)^{1/2}<\infty.
\]
This is true because writing
\[
A=\sup_x(1+|x|)^M|u(x)|,
\]
we have \(|u(x)|\le A(1+|x|)^{-M}\), and hence
\[
\|u\|_2^2
=\int_{\mathbb R^N}|u(x)|^2\,dx
\le A^2\int_{\mathbb R^N}(1+|x|)^{-2M}\,dx.
\]
Consequently,
\[
\begin{aligned}
\|\mathscr L_N^L g\|_2
&\le C_{N,M}
\sum_{|\alpha|+|\beta|\le2L}|a_{\alpha,\beta}|
\sup_x(1+|x|)^M
|x^\alpha\partial^\beta g(x)|\\
&\le C_{N,M,L}
\max_{|\beta|\le2L}
\sup_x(1+|x|)^{M+2L}|\partial^\beta g(x)|.
\end{aligned}
\]
The last expression is a continuous Schwartz seminorm.

\

Step 7. The estimates of the last step imply the expansion
\begin{equation}\label{eq:coupled-hermite}
 F=\sum_{b\in B}\sum_{\boldsymbol\nu\in\mathbb N_0^N}
 c_{b,\boldsymbol\nu}(F)
 \left(\bigotimes_{v\in S_\infty}h_{\nu_v}^{(n_v)}\right)
 \otimes e_b,
\end{equation}
with convergence in the fixed-support, fixed-level Schwartz
topology. To verify this assertion, let $q$ be any Schwartz
seminorm on $\cS(\R^N)$ and choose
$q(h_{\boldsymbol\nu})\leq C_q(1+|\boldsymbol\nu|)^{m_q}$.
There are
\[
 \#\{\boldsymbol\nu\in\mathbb N_0^N:
                   |\boldsymbol\nu|=j\}
 =\binom{j+N-1}{N-1}\ll_N(1+j)^{N-1}
\]
indices of total degree $j$. Hence, if $L>m_q+N$,
\[
 \sum_{\boldsymbol\nu}
 |c_{b,\boldsymbol\nu}(F)|q(h_{\boldsymbol\nu})
 \leq C_q\|\mathscr L_N^LF_b\|_{L^2}
       \sum_{\boldsymbol\nu}(1+|\boldsymbol\nu|)^{m_q-L}
 <\infty.
\]
Thus the series converges absolutely in every Schwartz seminorm.
Completeness of Schwartz space gives a Schwartz limit, and its
$L^2$ limit is $F_b$ by completeness of the Hermite basis.
The two limits agree because the inclusion
$\cS(\R^N)\hookrightarrow L^2(\R^N)$ is continuous.
This proves \eqref{eq:coupled-hermite}. In particular, its finite
partial sums form a dense subspace of each fixed-support,
fixed-level space.

\ 

Step 8. For $\eta\in\Omega$, define
\[
 \begin{aligned}
 A(\eta)(F)=
 \sum_{b\in B}\sum_{\boldsymbol\nu}
 c_{b,\boldsymbol\nu}(F)
 &\prod_{v\in S_\infty}
           A_v(\eta)(h_{\nu_v}^{(n_v)})\\
 &\times\prod_{v\in S_f}A_v(\eta)(e_{v,b_v}).
 \end{aligned}
\]
We now prove that this series defines a holomorphic distribution
family with the required bounds. Fix a compact $Q\subset\Omega$.
For every $v\in S_\infty$, the hypothesis supplies a continuous
Schwartz seminorm $q_{v,Q}$ and a constant $C_{v,Q}$ such that
\[
 \sup_{\eta\in Q}|A_v(\eta)(g)|
 \leq C_{v,Q}q_{v,Q}(g).
\]
The Hermite estimate therefore gives an exponent $m_v\geq0$
with
\[
 \sup_{\eta\in Q}
 |A_v(\eta)(h_{\nu_v}^{(n_v)})|
 \leq C'_{v,Q}(1+|\nu_v|)^{m_v}.
\]
For $v\in S_f$, the assumed bound on the chosen support and
translation level gives
\[
 \sup_{\eta\in Q}|A_v(\eta)(e_{v,b_v})|
 \leq C_{v,Q,C_v,M_v},
 \qquad b_v\in B_v,
\]
because $\|e_{v,b_v}\|_\infty=1$. Put
$m=\sum_{v\in S_\infty}m_v$. Multiplying these estimates yields
\[
 \sup_{\eta\in Q}
 \left|
 \prod_{v\in S_\infty}A_v(\eta)(h_{\nu_v}^{(n_v)})
 \prod_{v\in S_f}A_v(\eta)(e_{v,b_v})
 \right|
 \leq C_{Q,\mathbf C,\mathbf M}
                     (1+|\boldsymbol\nu|)^m.
\]
Choose an integer $L>m+N$. The coefficient estimate then proves
\[
 \begin{aligned}
 \sup_{\eta\in Q}|A(\eta)(F)|
 &\leq C_{Q,\mathbf C,\mathbf M}
     \sum_{b\in B}\|\mathscr L_N^LF_b\|_{L^2}
     \sum_{\boldsymbol\nu}(1+|\boldsymbol\nu|)^{m-L}\\
 &\leq C'_{Q,\mathbf C,\mathbf M}
                  \sum_{b\in B}\|\mathscr L_N^LF_b\|_{L^2}.
 \end{aligned}
\]
The last sum is a continuous seminorm on the chosen
fixed-support, fixed-level space.  Thus that series converges normally on
$Q$. Each term is holomorphic in $\eta$, because it is a finite
product of evaluations of the holomorphic families $A_v$ on
fixed test functions. The Weierstrass theorem
now shows that $\eta\mapsto A(\eta)(F)$ is holomorphic
on $\Omega$.

\ 

Step 9. The construction has the required tensor-product property. This and the equation \eqref{eq:coupled-hermite} show that the construction is unique
on each fixed-support, fixed-level space. They also prove
compatibility when the support is enlarged or the translation
level refined: the construction on the smaller space and the
restriction of the construction on the larger space are
continuous on the smaller space and agree on its dense subspace
of finite sums of decomposable tests. Any two choices of support
and level admit a common enlargement. The same argument proves independence
of the choice of representatives and Hermite bases.
Since Schwartz--Bruhat space has the locally convex inductive
limit topology of these spaces, $A(\eta)$ is a continuous
linear functional on the full Schwartz--Bruhat space for every
$\eta$, and the displayed bounds have exactly the required
uniformity on each fixed finite support and level.

\

Step 10. Finally, $p(\eta)=\prod_{v\in S}p_v(\eta)$ is holomorphic and
not identically zero on the connected polydisc $\Omega$.
Define, as a meromorphic family,
\[
 T(\eta)(F)=\frac{A(\eta)(F)}{p(\eta)}.
\]
The clearing factor is independent of $F$, including its finite
support and translation level. On the open set where $p\ne0$,
the tensor-product identity gives
\[
 T(\eta)\left(\bigotimes_{v\in S}f_v\right)
 =\prod_{v\in S}T_v(\eta)(f_v).
\]
Thus $T(\eta)$ is the desired tensor product there, and the
construction provides its meromorphic continuation to $\Omega$.
For two choices of local clearing factors $p,p'$ and cleared
families $A,A'$, the holomorphic functions
$p'(\eta)A(\eta)(F)$ and $p(\eta)A'(\eta)(F)$ agree wherever
$pp'\ne0$, first on decomposable tests and then on all tests by
density and continuity. The identity theorem makes them equal
on the overlap. Hence these local constructions patch to a
well-defined meromorphic distribution family on the parameter
domain. The uniform bounds for its cleared family were proved
above, and establish all the assertions of the lemma.
\end{proof}

\begin{lemma}\label{lem:coupled}
For arbitrary $F_{S_0}\in\cS(\prod_{v\in S_0}k_v^2)$ and fixed smooth local Whittaker functions $W_v$, $v\in S_0$,
which are right-$U_v$-finite at every archimedean place, the integral
\begin{equation}\label{eq:Hfactor}
 \begin{aligned}
 \cH_{S_0}(s)=\int_{\prod_{v\in S_0}(k_v^\times)^2}
 &F_{S_0}(x,y)\prod_{v\in S_0}
     W_v(a(1,x_vy_v))\\
 &\quad\times\prod_{v\in S_0}
   \omega_v(x_v)^{-2}\omega_v(y_v)^{-1}
   |x_v|_v^{3s-1}|y_v|_v^{s}
                     \,d^\times x_v\,d^\times y_v.
 \end{aligned}
\end{equation}
has meromorphic continuation to $\C$. The function $\cH_{S_0}(s)$ is holomorphic for $\Re s>13/64$.
\end{lemma}
\begin{proof}
Proposition~\ref{prop:local} and Lemma~\ref{lem:tensor-meromorphic}
give meromorphic continuation of the tensor distribution
$\bigotimes_{v\in S_0}T_{W_v}(s)$. In its domain of absolute
convergence it equals \eqref{eq:Hfactor}: this is Fubini's theorem
on decomposable test functions, followed by continuity and the
density used in \eqref{eq:coupled-hermite}. This proves meromorphy
for arbitrary $F_{S_0}$.

For holomorphy, \eqref{eq:localtube} and \eqref{eq:theta} give the
sufficient inequalities
\[
 3\Re s-1>-\frac12+\frac7{64},\qquad
 \Re s>-\frac12+\frac7{64}.
\]
The first is $\Re s>13/64$ and implies the second. On a compact
subset of this half-plane, choose slightly smaller positive
convergence margins in the proof of Proposition~\ref{prop:local}.
An arbitrary coupled Schwartz function is bounded by a product of
sufficiently large polynomial-decay weights at the archimedean
places and compact-support indicators at the finite places, with
a constant bounded by a Schwartz seminorm on each fixed finite
level. The product of the local integrable bounds then proves
normal convergence of \eqref{eq:Hfactor}, including all parameter
derivatives. Thus the continued function is holomorphic there.
\end{proof}

\subsection{Unramified calculation}
The double Hecke series in this calculation is the local counterpart of the series used by Hough \cite[Lemma~3.5 and equations~(3.20)--(3.21)]{Hough}.

For $v\notin S_0$, write $q=q_v$ and let $\alpha,\beta$ be the normalized Satake parameters. Put $c=\omega_v(\varpi_v)$. The central transformation law gives $\alpha\beta=c$, with $|c|=1$. The spherical Whittaker formula was obtained for $\GL_n$ by Shintani \cite[Theorem]{ShintaniWhittaker} and in the general unramified setting by Casselman--Shalika \cite[Theorem~5.4]{CS}. With our lower convention it gives
\begin{equation}\label{eq:spherical}
 W_v^\circ(a(1,\varpi^\ell))=q^{-\ell/2}h_\ell,
 \qquad h_\ell=\sum_{a=0}^{\ell}\alpha^{\ell-a}\beta^a
       \quad(\ell\geq0),
\end{equation}
with zero values for $\ell<0$. Therefore, with $X=c^{-2}q^{-u_1-1/2}$ and $Y=c^{-1}q^{-u_2-1/2}$, the unramified local integral is
\[
 \sum_{m,n\geq0}h_{m+n}X^mY^n.
\]
If $H(T)=\sum_{\ell\geq0}h_\ell T^\ell=(1-\alpha T)^{-1}(1-\beta T)^{-1}$, then
\begin{align}
 \sum_{m,n\geq0}h_{m+n}X^mY^n
 &=\frac{XH(X)-YH(Y)}{X-Y}\notag\\
 &=\frac{1-cXY}
       {(1-\alpha X)(1-\beta X)(1-\alpha Y)(1-\beta Y)}.
       \label{eq:generating}
\end{align}
For the first equality, group terms with $m+n=\ell$ and use
\[
 \sum_{m=0}^{\ell}X^mY^{\ell-m}
 =\frac{X^{\ell+1}-Y^{\ell+1}}{X-Y}.
\]
For the second equality put $A=\alpha+\beta$ and use $\alpha\beta=c$. The common numerator before division by $X-Y$ is
\begin{align*}
 X(1-AY+cY^2)-Y(1-AX+cX^2)
 &=X-Y+cXY^2-cX^2Y\\
 &=(X-Y)(1-cXY).
\end{align*}
The identity holds first as a power-series identity; the expression at $X=Y$ is interpreted by its removable limit. It also applies when $\alpha=\beta$. We obtain the following proposition: 
\begin{proposition}\label{prop:unramified} In a sufficiently far right chamber, 
    \begin{equation}\label{eq:localEuler}
 \cZ_v(\one_{\OO_v^2},W_v^\circ;u_1,u_2)
       =\frac{L_v(u_1+1/2,\pi_v\otimes\omega_v^{-2})
                   L_v(u_2+1/2,\pi_v\otimes\omega_v^{-1})}
                    {L_v(u_1+u_2+1,\omega_v^{-2})}.
\end{equation}
\end{proposition}

\subsection{Global result}
In a sufficiently far right chamber, multiplication of \eqref{eq:localEuler} for $v\not\in S_0$ and the remaining local integrals is justified by absolute convergence.  Consequently:
\begin{proposition}\label{prop:M-holo}
Initially for $\Re w>2$,
    \begin{equation}\label{eq:global-specialization}
 \begin{aligned}
 &\cZ(F_j,W_j;3w-1,w)\\
 &\quad=\frac{L^{S_0}(3w-1/2,\pi\otimes\omega^{-2})
             L^{S_0}(w+1/2,\pi\otimes\omega^{-1})}
              {L^{S_0}(4w,\omega^{-2})}\,\cH_{S_0,j}(w).
 \end{aligned}
\end{equation}
The function $A_j(w)= \cZ(F_j,W_j;3w-1,w)$ extends meromorphically to $\C$ and is holomorphic for $\Re w>1/4$.  
\end{proposition}
\begin{proof}
Absolute convergence in Lemma~\ref{lem:domination} justifies the
Euler product in the initial half-plane. Lemma~\ref{lem:coupled}
and Proposition~\ref{prop:unramified} then give the displayed
factorization and its meromorphic continuation.

The standard completed $L$-function of a cuspidal representation of
$\GL_2(\A)$ is entire by Jacquet--Langlands
\cite[Theorem~11.1]{JL}. Removing archimedean factors multiplies it
by reciprocal gamma factors, which are entire. Removing the
finitely many finite factors in $S_0$ multiplies it by polynomials
in $q_v^{-w}$. Unitary determinant twists preserve cuspidality, so this applies to
both numerator representations in \eqref{eq:global-specialization}.
For $\Re w>1/4$, the Euler product for
$L^{S_0}(4w,\omega^{-2})$ converges absolutely and is nonzero, while
$\cH_{S_0,j}(w)$ is holomorphic because $1/4>13/64$.
This proves the holomorphy assertion.
\end{proof}

We finally arrive at the following theorem: 

\begin{theorem}\label{thm:J-holo}
    The function $J(s, \Phi, \varphi)$ has an analytic continuation to a meromorphic function on $\C$. It is holomorphic for $\Re s > 1/4$. 
\end{theorem}
\begin{proof}
    This is a consequence of Proposition \ref{prop:tail}, Proposition \ref{prop:unfold}, and Proposition \ref{prop:M-holo}. 
\end{proof}

\section{Poisson summation formula and completion of the proof of Theorem \ref{thm:main}}\label{sec:poisson}

We use the alternating pairing in the adelic theory of binary cubic forms \cite{Wright}, with self-dual Haar measure. The pairing is
\begin{equation}\label{eq:pairing}
 [x,y]=x_1y_4-\frac13x_2y_3+\frac13x_3y_2-x_4y_1
\end{equation}
and we equip $V_\A$ with the Haar measure self-dual for $(x,y)\mapsto\psi_k([x,y])$. Define
\[
 \widehat\Phi(y)=\int_{V_\A}\Phi(x)\psi_k([x,y])\,dx.
\]

A direct check on diagonal and unipotent matrices gives
\begin{equation}\label{eq:pairing-covariance}
 [gx,gy]=\det(g)[x,y],\qquad
 \det\rho(g)=(\det g)^2.
\end{equation}
Put
\begin{equation}\label{eq:iota}
 g^\iota=(\det g)^{-1}g.
\end{equation}
The first identity in \eqref{eq:pairing-covariance} implies
\[
 [g^{-1}x,y]=[x,g^\iota y],
\]
so change of variables in the Fourier transform yields
\begin{equation}\label{eq:Fourier-covariance}
 \widehat{\Phi\circ\rho(g)}(y)
       =|\det g|_{\A}^{-2}\widehat\Phi(g^\iota y).
\end{equation}
The map $g\mapsto g^\iota$ is an involution and takes $G_k$ to itself. It preserves Haar measure. Also
\begin{equation}\label{eq:iota-determinant}
 |\det g^\iota|_{\A}=|\det g|_{\A}^{-1}.
\end{equation}
For a form of central character $\omega$, define
\begin{equation}\label{eq:automorphic-iota}
 \varphi^\iota(h)=\varphi(h^\iota)
       =\omega(\det h)^{-1}\varphi(h).
\end{equation}
This is a smooth $U_\A$-finite automorphic form in
$\pi^\iota=\pi\otimes\omega^{-1}$, with central character
$\omega^{-1}$. It is cuspidal whenever $\varphi$ is, since
$n(c)^\iota=n(c)$. Moreover $(\varphi^\iota)^\iota=\varphi$ and
\begin{equation}\label{eq:iota-inverse-value}
 \varphi((h^\iota)^{-1})=\varphi^\iota(h^{-1}).
\end{equation}
Thus Poisson duality exchanges the central characters $\omega$ and
$\omega^{-1}$.

Define
\begin{equation}\label{eq:Zplus}
 Z_+(s,\Phi,\varphi)
 =\int_{\substack{G_\A/G_k\\|\det g|_\A\geq1}}
 |\det g|_\A^s\varphi(g^{-1})
       \sum_{x\in V'_k}\Phi(gx)\,dg.
\end{equation}
Also let 
\begin{equation}\label{eq:I-2}
I(s,\Phi,\varphi)
 =\int_{\substack{G_\A/G_k\\|\det g|_\A\leq1}}
       |\det g|_\A^s\varphi(g^{-1})
         \sum_{x\in S_k}\Phi(g^\iota x)\,dg.
         \end{equation}
\begin{lemma}\label{lem:Z+I}
    The function $Z_+(s,\Phi,\varphi)$ is entire, and $I(s,\Phi,\varphi)$ has an analytic continuation to an entire function. 
\end{lemma}

\begin{proof}
The integral $Z_+$ is entire by Proposition~\ref{prop:tail}. 
In \eqref{eq:I-2}, initially for $\Re t>0$, substitute $h=g^\iota$. The resulting expression is a large-determinant singular integral with weight $|\det h|_{\A}^{-t}$. Its automorphic coefficient is $\varphi^\iota(h^{-1})$. Its zero orbit is absolutely integrable when $\Re t>0$ and vanishes by the zero average of $\varphi^\iota$ on each determinant fiber. Its triple-root orbit likewise vanishes, by the absolutely convergent unipotent unfolding in \S\ref{sec:radial}. Its nonzero terms have arbitrary radial decay. Therefore
\begin{equation}\label{eq:I-entire}
 I(t,\Phi,\varphi)=E_{\Phi,\varphi^\iota}(-t)\qquad(\Re t>0),
\end{equation}
where $E_{\Phi,\varphi^\iota}$ denotes \eqref{eq:E}. This identity continues $I$ to an entire function. 
\end{proof}

\begin{lemma}\label{lem:7}
For a cusp form $\varphi$ satisfying the hypotheses of
Theorem~\ref{thm:main}, a Schwartz--Bruhat function $\Phi$, and
$\Re s>2$, one has
\begin{equation}\label{eq:Poisson}
 \begin{split}
 Z(s,\Phi,\varphi)
 ={}&Z_+(s,\Phi,\varphi)+Z_+(2-s,\widehat\Phi,\varphi^\iota)\\
 &+I(s-2,\widehat\Phi,\varphi)-J(s,\Phi,\varphi).
 \end{split}
\end{equation}
\end{lemma}
\begin{proof}
The theta estimates of Lemma~\ref{lem:theta} justify splitting the
nonsingular integral at determinant norm one. On its small-determinant
part, add the singular contribution $J$ to obtain
\begin{equation}\label{eq:poisson-small-domain}
 \begin{split}
 Z(s,\Phi,\varphi)+J(s,\Phi,\varphi)
 ={}&Z_+(s,\Phi,\varphi)\\
 &+\int_{|\det g|_\A\leq1}|\det g|_\A^s\varphi(g^{-1})
                         \sum_{x\in V_k}\Phi(gx)\,dg.
 \end{split}
\end{equation}
The integral is over the right automorphic quotient. In particular,
no full radial singular integral is introduced.
Poisson summation and \eqref{eq:Fourier-covariance} give, pointwise,
\[
 \sum_{x\in V_k}\Phi(gx)
   =|\det g|_\A^{-2}\sum_{x\in V_k}\widehat\Phi(g^\iota x).
\]
After inserting this identity, split the dual sum into $V'_k$ and
$S_k$. The latter contribution is exactly
$I(s-2,\widehat\Phi,\varphi)$; it is absolutely convergent here
because $\Re(s-2)>0$. In the former contribution, put $h=g^\iota$.
Equations \eqref{eq:iota-determinant} and
\eqref{eq:radial-measure}, together with \eqref{eq:iota-inverse-value}, give
\[
 \begin{split}
 &\int_{|\det g|_\A\leq1}|\det g|_\A^{s-2}\varphi(g^{-1})
                      \sum_{x\in V'_k}\widehat\Phi(g^\iota x)\,dg\\
 &\qquad=\int_{|\det h|_\A\geq1}|\det h|_\A^{2-s}
              \varphi^\iota(h^{-1})\sum_{x\in V'_k}\widehat\Phi(hx)\,dh\\
 &\qquad=Z_+(2-s,\widehat\Phi,\varphi^\iota).
 \end{split}
\]
This proves \eqref{eq:Poisson}, with every interchange performed in
an absolutely convergent integral. The argument is the small- and
large-determinant decomposition in Wright's adelic Poisson method
\cite{Wright}.
\end{proof}

\begin{proof}[Proof of Theorem~\ref{thm:main}]
Lemma~\ref{lem:Z+I} makes both large-determinant terms and the dual singular term in \eqref{eq:Poisson} entire. Theorem~\ref{thm:J-holo} gives meromorphy of $J$ on $\C$ and holomorphy for $\Re s>1/4$. The same conclusions therefore hold for $Z$. The factorization and the entire tail in Theorem~\ref{thm:main} were established in Sections~\ref{sec:radial}--\ref{sec:factorization}; the identity theorem gives uniqueness of all continuations.
\end{proof}

\end{document}